\documentclass[hidelinks,onefignum,onetabnum]{siamart251216} 

\usepackage{lipsum}
\usepackage{amsfonts}
\usepackage{graphicx}
\usepackage{epstopdf}
\ifpdf
  \DeclareGraphicsExtensions{.eps,.pdf,.png,.jpg}
\else
  \DeclareGraphicsExtensions{.eps}
\fi

\newsiamremark{remark}{Remark}
\newsiamremark{example}{Example}
\newsiamremark{hypothesis}{Hypothesis}
\crefname{hypothesis}{Hypothesis}{Hypotheses}
\newsiamthm{claim}{Claim}
\newsiamremark{fact}{Fact}
\crefname{fact}{Fact}{Facts}

\headers{A Unique Inverse Decomposition of SPD Matrices}{Y. Dolinsky and O. Zuk}

\title{A Unique Inverse Decomposition of Positive Definite Matrices under Linear Constraints\thanks{Submitted to SIMAX.
\funding{This work was funded by the ISF 305/25 and ISF 2392/22}}}

\author{Yan Dolinsky\thanks{Department of Statistics and Data Science, The Hebrew University, Jerusalem, Israel 
  (\email{Yan.Dolinsky@mail.huji.ac.il}, \email{or.zuk@mail.huji.ac.il}).}
\and Or Zuk\footnotemark[2]}

\usepackage{amsopn}
\DeclareMathOperator{\diag}{diag}

\usepackage[utf8]{inputenc}
\usepackage[T1]{fontenc}
\usepackage{amsmath}
\usepackage{amssymb}
\usepackage{comment}
\usepackage{geometry}
\usepackage{mathtools}
\usepackage{algorithm}
\usepackage{algpseudocode}
\usepackage{float}   
\usepackage{tabularx}
\usepackage{booktabs}
\usepackage{array}
\usepackage{hyperref}
\usepackage{subcaption}
\usepackage{color}
\usepackage{cite}

\usepackage{xr}

\newcommand{\Sub}{{\mathcal{S}}}   
\newcommand{\Orth}{{\mathcal{O}}}  
\newcommand{\band}{{\mathcal{D}}}  
\newcommand{\cov}{{\Sigma}} 
\newcommand{\pre}{{\Lambda}} 

\newcommand{\calU}{\mathcal U}
\newcommand{\calL}{\mathcal L}
\newcommand{\calG}{\mathcal G}   
\newcommand{\calS}{{\mathcal S}_{2n}}  
\newcommand{\calM}{\mathcal M}

\DeclareMathOperator{\proj}{proj}
\DeclareMathOperator*{\argmin}{arg\,min}

\algtext*{EndFor}
\algtext*{EndIf}
\algtext*{EndWhile}
\algtext*{EndLoop}
\algtext*{EndFunction}

\DeclareMathOperator{\tr}{tr}

\algrenewcommand\algorithmicrequire{\textbf{Input:}}
\algrenewcommand\algorithmicensure{\textbf{Output:}}

\algnewcommand{\Parameters}{%
	\Statex\noindent\hspace*{-1.7em}\textbf{Parameters:}\hspace*{0.5em}%
}

\newcommand{\be}{\begin{equation}}
\newcommand{\ee}{\end{equation}}

\newcommand{\R}{\mathbb{R}}
\newcommand{\Sym}{\mathbb{S}^n} 
\newcommand{\SPD}{\mathbb{S}^n_{++}} 
\newcommand{\SPSD}{\mathbb{S}^n_{+}} 

\ifpdf
\hypersetup{
  pdftitle={A Unique Inverse Decomposition of Positive Definite Matrices under Linear Constraints},
  pdfauthor={Y. Dolinsky and O. Zuk}
}
\fi

\begin{document}

\maketitle

\begin{abstract}
We study a nonlinear decomposition of a positive definite matrix into two components: 
the inverse of another positive definite matrix and a symmetric matrix constrained to lie 
in a prescribed linear subspace. Equivalently, the matrix being inverted is required to lie 
in the orthogonal complement of that subspace with respect to the trace inner product. 
Under a sharp nondegeneracy condition on the subspace, we show that every positive definite 
matrix admits a \emph{unique} decomposition of this form.

This decomposition admits a variational characterization as the unique minimizer of a 
strictly convex log-determinant optimization problem, which in turn yields a natural dual 
formulation that can be efficiently exploited computationally. We derive several properties, 
including the stability of the decomposition.

We further develop feasibility-preserving Newton-type algorithms with provable convergence 
guarantees and analyze their per-iteration complexity in terms of algebraic properties of 
the decomposed matrix and the underlying subspace. Finally, we present two applications: one in mathematical finance, where the decomposition arises naturally in exponential utility maximization, and one in linear algebra, where our main theorem characterizes a new bilinear triangular decomposition for positive definite matrices, obtained together with a direct elimination algorithm.
\end{abstract}

\begin{keywords}
positive definite matrices, inverse-based matrix decomposition, bilinear triangular decomposition, log-determinant optimization, Newton-CG, exponential utility maximization
\end{keywords}

\begin{MSCcodes}
15A39, 15B48, 90C25, 91G10
\end{MSCcodes}

\section{Introduction}
\label{sec:introduction}

We study a nonlinear decomposition of a symmetric positive definite (SPD) matrix $A \in \SPD$ of the form $A = (B^{*})^{-1} + C^{*}$, where the components are constrained such that $C^{*} \in \Sub$ and $B^{*} \in \Sub^{\perp} \cap \SPD$ for a prescribed linear subspace $\Sub \subset \Sym$. Our main contribution is a global existence and uniqueness theorem for this decomposition, valid for any subspace $\Sub$ satisfying the sharp nondegeneracy condition $\Sub \cap \SPSD = \{0\}$. By intertwining matrix inversion with these dual linear constraints, the result provides a canonical algebraic decomposition that serves as a structured, non-Euclidean analogue to the classical orthogonal decomposition of symmetric matrices.

Our proof proceeds by characterizing the decomposition as the unique stationary point of a strictly convex log-determinant functional. Log-determinant objectives are classical and arise throughout convex analysis, matrix theory, information geometry, and statistics (see, e.g., \cite{boyd2004convex, cover2006elements, amari2016information, vandenberghe1998determinant}); nonlinear projections under Bregman divergences were studied in \cite{dhillon2008matrix}; and closest to our setting, Gaussian maximum-likelihood estimation under linear constraints on the precision matrix, from classical covariance selection \cite{dempster1972covariance} to models with general linear constraints on the concentration matrix \cite{sturmfels2009multivariate,uhler2012geometry,zwiernik2014linear}, has an objective coinciding with the dual problem of Section~\ref{subsec:duality_general}. Here, however, the optimization problem serves as a \emph{tool} rather than the primary object of study. The resulting theorem is an exact algebraic statement about matrix structure: for every admissible subspace $\Sub$, the decomposition exists globally, is unique, and depends smoothly on $A$. The feasibility condition on $\Sub$ is shown to be both natural and necessary.

Beyond existence and uniqueness, we establish several structural properties.
The decomposition inherits any symmetry shared by $A$ and $\Sub$ under a finite
orthogonal group, which reduces both the effective dimension and the cost of
computing it. We also derive determinant inequalities, relate the decomposition
of $A$ to that of $A^{-1}$, and prove that the decomposition is stable under
perturbations of $A$.

From an algorithmic perspective, the variational characterization leads
naturally to Newton-type methods with feasibility-preserving line search: we
give explicit gradients and Hessians, prove global convergence with local
quadratic rates under the standard inexact-Newton forcing conditions, and
establish per-iteration complexity bounds. Which method is preferable, and what
each iteration costs, is governed by the algebraic structure of the instance,
and we analyze this in several settings: banded and circulant matrices, where
the structure is preserved along the optimization path; the dual formulation,
which is advantageous when $\dim(\Sub^{\perp})$ is small; and group-invariant
problems, where symmetry reduces the effective dimension
(see, e.g., \cite{gray2006toeplitz,bhatia2013matrix}).

Finally, we illustrate the theory with two applications.
First, we show that the decomposition arises naturally in a Gaussian
exponential-utility maximization problem in mathematical finance,
where $B^{*}$ and $C^{*}$ admit clear economic interpretations
related to information constraints and hedging structure.
Second, we establish a new
result in linear algebra: for a pair of positive definite matrices
$A$ and $\pre$ we obtain the bilinear triangular decomposition
$\pre = L + U + UAL$, with $U$ strictly upper triangular and $L$ lower
triangular, and give
necessary and sufficient conditions for its existence and uniqueness.
We also develop a direct elimination algorithm,
analogous in spirit to LU decomposition, together with its complexity
analysis. This $U$ is
exactly the component our main theorem produces from an associated
$2n\times2n$ matrix, for every positive definite pair $(A,\pre)$. The
identification therefore holds even when the triangular decomposition does not
exist, and this is what yields the criterion for when it does.

\smallskip
\noindent
\textbf{Outline.}
Section~\ref{sec:general_decomposition} presents the main decomposition theorem and its proof. In this section we also derive several useful properties of the decomposition. Section~\ref{sec:optimization} develops algorithms together with complexity guarantees. Section~\ref{sec:finance} presents an application to a utility maximization problem in finance. Section~\ref{sec:linear_algebra} applies the decomposition to derive a new bilinear triangular decomposition for positive definite matrices and develops an associated elimination algorithm for a structured matrix class. Finally, Section~\ref{sec:conclusion} concludes the paper and discusses possible directions for future research.

\section{Decomposition Theorem}
\label{sec:decomposition_theorem}
\label{sec:general_decomposition}

Let $\Sym$ denote the space of symmetric real $n \times n$ matrices, equipped
with the Frobenius inner product $\langle X,Y\rangle_F:=\tr(XY)$ and its induced
norm $\|\cdot\|_F$. Let $\SPD \subset \Sym$ be the (open) cone of
positive-definite matrices, and let $\SPSD \subset \Sym$ be the (closed) cone of
positive semidefinite matrices. A linear subspace $\Sub\subset\Sym$ has
orthogonal complement
$\Sub^\perp := \{\,X\in\Sym:\ \tr(XY)=0\ \ \forall\,Y\in \Sub\,\}$.
We arrive at the main result of the paper.

\begin{theorem}[General structured decomposition]
\label{thm:general_decomposition}
Let $\Sub\subset\Sym$ be a linear subspace.
\begin{enumerate}
\item[\textbf{(i)}] If $\Sub$ contains no nonzero positive semidefinite matrix,
that is,
\be
\Sub \cap \SPSD = \{0\},
\label{eq:nopsddirections}
\ee
then every $A\in\SPD$ admits a \emph{unique} pair
$(B^\ast \in \SPD \cap \Sub^{\perp}, C^\ast \in \Sub)$ such that
\be\label{eq:maindecomposition}
A = {B^\ast}^{-1} + C^\ast.
\ee
\item[\textbf{(ii)}] Conversely, if \eqref{eq:maindecomposition} admits a
solution for even a single $A\in\SPD$, then \eqref{eq:nopsddirections} holds.
\end{enumerate}
\end{theorem}

Lemma~\ref{lem:bounded_intersection} below shows that the set of matrices $C$
yielding a positive definite $A-C$ is bounded. The proof of
Theorem~\ref{thm:general_decomposition} then frames the decomposition as the
minimization of the log-determinant objective \(C\mapsto-\log\det(A-C)\), whose
stationarity condition is exactly the trace condition on $B$.

\begin{lemma}
\label{lem:bounded_intersection}
Let $A \in \mathbb{S}_{++}^n$ and let $\Sub \subset \mathbb{S}^n$ be a linear subspace satisfying \eqref{eq:nopsddirections}. 
Define the feasible set
\be
\Omega_C := \{\, C \in \Sub : A - C \succ 0 \,\}.
\label{eq:feasible_matrices}
\ee
Then $\Omega_C$ is a bounded subset of $\Sub$. In particular, its closure $\overline{\Omega_C}$ is compact.
\end{lemma}

\begin{proof}
Assume by contradiction that $\Omega_C$ is unbounded. Then there exists a
sequence $\{C^{(k)}\}_{k=1}^{\infty} \subset \Omega_C$ with $\|C^{(k)}\|_F \underset{k \to \infty}{\longrightarrow} \infty$, where $|| \cdot ||_F$ is the Frobenius norm. Define the normalized matrices
\be
X^{(k)} := \frac{C^{(k)}}{\|C^{(k)}\|_F} \in \Sub.
\ee
Since $\|X^{(k)}\|_F = 1$, the sequence $\{X^{(k)}\}_{k=1}^{\infty}$ lies in a compact set, and therefore it has a converging subsequence $\{X^{(k_j)}\}_{j=1}^{\infty}$ with 
$X^{(k_j)} \underset{j \to \infty}{\longrightarrow} X$ for some $X \in \Sub$ with $\|X\|_F = 1$.
Take the corresponding subsequence $\{C^{(k_j)}\}_{j=1}^{\infty}$ of $\{C^{(k)}\}_{k=1}^{\infty}$. The condition $A - C^{(k_j)} \succ 0$ yields for every fixed nonzero vector $v \in \mathbb{R}^n$,
\be
\frac{v^{\top} C^{(k_j)} v}{\|C^{(k_j)}\|_F}
< 
\frac{v^{\top} A v}{\|C^{(k_j)}\|_F} \underset{j \to \infty}{\longrightarrow} 0,
\label{eq:quad_form_condition}
\ee

where the limit on the right-hand side follows from the assumption that
$\|C^{(k)}\|_F \underset{k \to \infty}{\longrightarrow} \infty$. But continuity of quadratic forms yields
\be
\frac{v^{\top} C^{(k_j)} v}{\|C^{(k_j)}\|_F}
\; \underset{j \to \infty}{\longrightarrow} \;
v^{\top} X v .
\ee
Hence $v^{\top} X v \le 0$ for all $v \in \mathbb{R}^n$. Thus, $X \preceq 0$. But $X\neq 0$ and $X \in \Sub$, and so $-X \in \Sub$ satisfies $-X \succeq 0$, in 
contradiction to $\Sub \cap \SPSD = \{0\}$.
Therefore $\Omega_C$ must be bounded.
\end{proof} 

\begin{proof}
\label{proof:main_decomposition}

The proof is divided into three steps. Steps 1 and 2 prove part~\textbf{(i)},
so throughout them we assume \eqref{eq:nopsddirections} and fix $A\in\SPD$.
Step~3 proves part~\textbf{(ii)}.

\smallskip
\emph{Step 1: Reduction of the decomposition to a stationarity condition.}
Let $D_1,...,D_m\in \Sub$ be a basis for $\Sub$.
Define the linear bijection
\be
C(x) := \sum_{k=1}^m x_k D_k , \qquad x \in \mathbb{R}^m,
\ee
and the corresponding feasible parameter set
\be
\Omega_x := \{\, x \in \mathbb{R}^m : A - C(x) \succ 0 \,\}.
\label{eq:x_feasible_set}
\ee
Hence the feasible set of matrices $\Omega_C$ defined in eq. \eqref{eq:feasible_matrices} satisfies
$\Omega_C = C(\Omega_x)$.

For $x \in \Omega_x$ let
\be
M(x) := A - C(x),
\qquad
B(x) := M(x)^{-1},
\qquad
\Phi(x) := -\log\det\bigl(M(x)\bigr).
\label{eq:phi_x}
\ee
Using the chain rule for $M$ and $-\log\det(\cdot)$ we obtain
\be
\label{eq:grad_phi}
\frac{\partial \Phi}{\partial x_k}(x)
= \tr \big( B(x) D_k \big), 
\qquad k = 1,\dots,m.
\ee

Thus a pair $(B,C)\in \SPD \times \Sub$ with $A=B^{-1}+C$ and $\tr(BD_k)=0$ for all $k$ corresponds 
precisely to a parameter $x$ such that $C=C(x)$ and 
\be
\nabla\Phi(x)=0.
\ee

\emph{Step 2: Existence and uniqueness of a minimizer.}
The map $M(x)$ is affine and injective, and $-\log\det(\cdot)$ is strictly convex on 
$\SPD$, hence $\Phi$ is strictly convex on the convex set $\Omega_x$.
If a sequence $\{x^{(k)}\}_{k=1}^{\infty}$ approaches $\partial\Omega_x$, then
along a subsequence $x^{(k)}\to\bar x\in\partial\Omega_x$, so that
$M(x^{(k)})\to M(\bar x)\succeq0$ with $M(\bar x)$ singular; hence
$\det\bigl(M(x^{(k)})\bigr) \to 0$ and $\Phi(x^{(k)}) \to +\infty$ as $k \to \infty$.

Since $A \succ 0$ and $C(0)=0$, we have $0 \in \Omega_x$ and $\Omega_x$ is nonempty.
By Lemma~\ref{lem:bounded_intersection}, the set $\Omega_C$ is bounded. 
Thus $\Omega_x = C^{-1}(\Omega_C)$ is bounded as well, and $\overline{\Omega_x}$ 
is compact.

Extend $\Phi$ to $\overline{\Omega_x}$ by setting $\Phi = +\infty$ on 
$\partial\Omega_x$.
Lower semicontinuity of $\Phi$ and compactness of $\overline{\Omega_x}$ yield 
a minimizer $x^\ast \in \Omega_x$, and strict convexity implies that 
$x^\ast$ is unique. Since $\Phi$ is differentiable and strictly convex on the open set $\Omega_x$, any stationary point is the global minimizer. Thus the unique 
minimizer $x^\ast$ is also the unique point in $\Omega_x$ satisfying 
$\nabla\Phi(x^\ast)=0$.

Define $C^\ast := C(x^\ast)$ and $B^\ast := (A - C^\ast)^{-1}$.
Then the pair $(B^\ast,C^\ast) \in \SPD \times \Sub$ satisfies \eqref{eq:maindecomposition}.
Uniqueness of $x^\ast$ and the equivalence in Step~1 imply the uniqueness of the decomposition
\eqref{eq:maindecomposition}. This proves part~\textbf{(i)}.

\smallskip
\emph{Step 3: Necessity of \eqref{eq:nopsddirections}.}
For part~\textbf{(ii)}, suppose \eqref{eq:maindecomposition} holds for some
$A\in\SPD$, with $B^\ast\in\SPD\cap \Sub^\perp$ and $C^\ast\in \Sub$, and let
$0\neq X\in \Sub\cap\SPSD$. Then $B^\ast\in \Sub^\perp$ gives $\tr(B^\ast X)=0$,
whereas $B^\ast\succ0$ together with $X\succeq0$, $X\neq0$ gives
$\tr(B^\ast X)\ge\lambda_{\min}(B^\ast)\tr(X)>0$, a contradiction. Hence
$\Sub\cap\SPSD=\{0\}$.
\end{proof}

\begin{remark}\label{remark1}
A sufficient condition for \eqref{eq:nopsddirections} is the existence of a
matrix $\hat B \in \SPD$ such that
\[
\tr(\hat B X) = 0, \ \forall X\in \Sub.
\]

Part~\textbf{(ii)} of Theorem~\ref{thm:general_decomposition} has a variational counterpart, which explains how the construction of Step~2 fails without \eqref{eq:nopsddirections}: the condition is also necessary for a minimizer of $-\log\det(A-C)$ over $\Omega_C$ to exist. Suppose there exists a matrix $X \neq 0$ satisfying $X \in \SPSD \cap \Sub$. Define $C_t = -tX$. Then $A-C_t = A + tX \succ 0$ for all $t \geq 0$, hence the feasible set $\Omega_C$ is unbounded. Therefore, $\det(A-C_t) \underset{t \to \infty}{\longrightarrow} \infty$ hence the objective $-\log\det(A-C_t)$ is unbounded below.
\end{remark}

\subsection{A dual variational formulation}
\label{subsec:duality_general}

The decomposition of Theorem~\ref{thm:general_decomposition} admits a natural
primal-dual interpretation. This viewpoint clarifies the variational structure
of the problem and suggests alternative algorithmic approaches. The dual
formulation below coincides formally with the Gaussian maximum-likelihood
objective under linear constraints on the precision matrix, but is used here in
the context of a deterministic decomposition.

\paragraph{Primal formulation}
By Step~2 of the proof of Theorem~\ref{thm:general_decomposition}, the matrix
$C^\ast$ is the unique minimizer of the strictly convex problem
\be
\min_{C\in\Omega_C} \; -\log\det(A-C) .
\label{eq:primal_logdet}
\ee

\paragraph{Dual formulation}
The primal problem \eqref{eq:primal_logdet} can be written as minimizing $-\log\det(M)$ over $M\succ0$ subject to the linear constraint $A-M\in \Sub$, where $M=A-C$ is an auxiliary variable.
The associated Lagrangian involves a dual variable $B\in \Sub^\perp$.
Eliminating $M$ via the first-order optimality condition $M^{-1}=B$ yields the
dual problem stated below:

Let $(B^\ast \in \SPD \cap \Sub^{\perp},\, C^\ast \in \Sub)$ be the unique pair satisfying
\eqref{eq:maindecomposition}. For any $B \in \Sub^{\perp}$ and $C \in \Sub$ we have $\tr(B C) = 0$ hence $\tr \big( (A-C)B \big)=\tr(BA)$. 
The function
\[
\Psi_A(B):= \log\det(B)-\tr(BA)+n, \qquad B\in \SPD \cap \Sub^{\perp},
\]
is strictly concave and satisfies the first-order optimality condition
\[
\nabla\Psi_A(B^\ast) = {B^\ast}^{-1}-A \quad ; \quad \tr\!\big(( {B^\ast}^{-1}-A )X\big)=0 \:,\: \forall X\in \Sub^\perp .
\]
Moreover, using the decomposition $A={B^\ast}^{-1}+C^\ast$ with $C^\ast\in \Sub$, we have for every $X\in \Sub^\perp$, 
\[
\tr\big( ({B^\ast}^{-1}-A) X \big) = -\tr(C^\ast X)=0 . 
\]

Hence $B^\ast$ satisfies the first-order optimality conditions for $\Psi_A$.
Since $\log\det(B)$ is strictly concave on $\SPD \cap \Sub^\perp$, we conclude that $B^\ast$ is the unique maximizer of the dual problem
\be
\label{eq:dual_logdet}
\max_{B\succ 0,\; B\in \Sub^{\perp}} \log\det(B)-\tr(BA)+n.
\ee

Since the primal problem \eqref{eq:primal_logdet} is convex and strictly feasible,
strong duality holds, and the primal and dual optimal values coincide. Accordingly, the dual problem \eqref{eq:dual_logdet} provides an alternative computational
route to the decomposition.
One may optimize directly over $B \in \Sub^\perp \cap \SPD$ using standard iterative methods, including Newton-type schemes with feasibility-preserving backtracking.
The primal formulation is preferable when $m=\dim \Sub$ is small, while the dual
formulation may be preferable when $\dim(\Sub^\perp)$ is comparatively small or when the
constraint structure is naturally expressed in terms of $B$.

\subsection{Properties of the decomposition}\label{sec:properties}

In this subsection we collect several properties of the
variational characterization of the decomposition
\eqref{eq:maindecomposition}.
\begin{enumerate}

\item \textbf{Action by conjugation.}
Let $\Orth(n)$ denote the set of all $n \times n$ orthogonal matrices.
For $P \in \Orth(n)$, consider the action on $\Sym$ by conjugation,
$X \mapsto P X P^\top$ ($P^\top$ denotes transposition). 
From the equality
\[
\tr\!\left(P B^\ast P^\top \, P X P^\top\right)
= \tr(B^\ast X) = 0 \:,
\qquad \forall X \in \Sub,
\]
we conclude that
\[
P A P^\top = (P B^\ast P^\top)^{-1} + P C^\ast P^\top
\]
is the unique decomposition \eqref{eq:maindecomposition} of $P A P^\top$
associated with the subspace
$
\tilde \Sub := P \Sub P^\top .
$

Next, we describe the following inheritance of group symmetry.
Let $\calG \subset \Orth(n)$ act on $\Sym$ by conjugation,
$
X \mapsto P X P^\top .
$
Assume that $A$ is $\calG$-fixed, i.e.,
$
PA P^\top = A \quad \forall P \in \calG,
$
and that $\Sub$ is $\calG$-invariant, i.e.,
$
P \Sub P^\top = \Sub \quad \forall P \in \calG.
$
By the above analysis and the uniqueness of the decomposition
\eqref{eq:maindecomposition}, the decomposition is then $\calG$-fixed:
\be
\label{eq:G_fixed}
P B^\ast P^\top = B^\ast,
\qquad
P C^\ast P^\top = C^\ast,
\quad \forall P \in \calG.
\ee

\item \textbf{Determinant inequality.}
From the primal representation,
\[
-\log\det(A) = -\log\det(A-0) \;\ge\; -\log\det(A-C^\ast),
\]
where the inequality follows from optimality of $C^\ast$.
Hence
\[
-\log\det(A) \ge -\log\det({B^\ast}^{-1}),
\]
which implies
\be
\label{eq:det_ineq}
\det({B^\ast}^{-1}) \ge \det(A),
\quad\text{i.e.}\quad
\det(B^\ast) \le \det(A^{-1}).
\ee

\item \textbf{Decomposition of the inverse.}
Define the transformed subspace
\[
\hat \Sub := A^{-1} \Sub A^{-1}
= \{\, A^{-1} X A^{-1} : X \in \Sub \,\} \subset \Sym,
\]
and set
\be\label{deco}
\hat B := A B^\ast A,
\qquad
\hat C := A^{-1} C^\ast A^{-1}.
\ee
Then, for all \(X \in \Sub\),
\[
\tr(\hat B\, A^{-1} X A^{-1}) = \tr(B^\ast X) = 0.
\]
Consequently,
\[
\hat B \in \hat \Sub^\perp, \quad
\hat B \succ 0, \quad
\hat C \in \hat \Sub, \quad
A^{-1} = \hat B^{-1} + \hat C.
\]
The subspace \(\hat \Sub\) inherits \eqref{eq:nopsddirections}: the congruence
\(X\mapsto A^{-1}XA^{-1}\) is a linear bijection of \(\Sym\) that preserves
positive semidefiniteness in both directions, so a nonzero element of
\(\hat \Sub\cap\SPSD\) would produce a nonzero element of \(\Sub\cap\SPSD\). By
uniqueness of the decomposition \eqref{eq:maindecomposition},
\((\hat B,\hat C) \in \SPD \times \hat \Sub\)
is the unique decomposition of \(A^{-1}\) associated with \(\hat \Sub\).

\end{enumerate}

\subsection{Stability under perturbations}
\label{subsec:stability}

By Theorem~\ref{thm:general_decomposition} each $A\in\SPD$ determines a unique
pair, which we write $\big(B^{*}(A),C^{*}(A)\big)$. This subsection bounds how
far that pair moves when $A$ is perturbed to $A+\Delta$, in terms of quantities
attached to the unperturbed problem. Assume throughout that $\Sub\neq\{0\}$
(for $\Sub=\{0\}$ the decomposition is $C^{*}(A)=0$, $B^{*}(A)=A^{-1}$). For
$B\in\SPD$ set
\be
\label{eq:alpha_def_fro}
\alpha(B,\Sub):=\min_{X\in \Sub,\ \|X\|_F=1}\tr(BXBX) ,
\ee
where $\tr(BXBX)$ is the second derivative of $C\mapsto-\log\det(A-C)$ in the
direction $X$, at the point $C$ with $B=(A-C)^{-1}$. Thus $\alpha(B,\Sub)$
measures how sharply that objective curves along $\Sub$: when it is small the
minimizer is poorly determined, and the decomposition is correspondingly
sensitive.

\begin{lemma}[Positivity and Lipschitz continuity of $\alpha$]
\label{lem:alpha_lipschitz}
Let $B,B'\in\SPD$ and let $\Sub\subset\Sym$ be a nonzero linear subspace.
\begin{enumerate}
\item[\textbf{(i)}] $0<\lambda_{\min}(B)^2\le\alpha(B,\Sub)\le\|B\|_2^2$.
\item[\textbf{(ii)}] $\big|\alpha(B',\Sub)-\alpha(B,\Sub)\big|\le\big(\|B'\|_2+\|B\|_2\big)\,\|B'-B\|_2$.
\end{enumerate}
\end{lemma}

\begin{proof}
\textbf{(i)} For $X\in\Sym$,
\begin{align*}
\tr(BXBX)&=\|B^{1/2}XB^{1/2}\|_F^2 ,\\
\lambda_{\min}(B)\|X\|_F&\le\|B^{1/2}XB^{1/2}\|_F\le\|B\|_2\|X\|_F .
\end{align*}
Squaring the second chain and minimizing over $X\in \Sub$ with $\|X\|_F=1$, a
compact set, gives the two inequalities, the first strict because
$\lambda_{\min}(B)>0$ for $B\in\SPD$.

\textbf{(ii)} For every $X\in\Sym$,
\begin{align*}
\big|\tr(B'XB'X)-\tr(BXBX)\big|
&=\big|\tr\big((B'-B)XB'X\big)+\tr\big(BX(B'-B)X\big)\big|\\
&\le\|(B'-B)X\|_F\big(\|B'X\|_F+\|BX\|_F\big)\\
&\le\big(\|B'\|_2+\|B\|_2\big)\,\|B'-B\|_2\,\|X\|_F^2 ,
\end{align*}
the second step by the triangle inequality and
$|\tr(MN)|\le\|M\|_F\|N\|_F$, the third by $\|MX\|_F\le\|M\|_2\|X\|_F$. At every
$X\in\Sub$ with $\|X\|_F=1$ the two traces therefore differ by at most
$\big(\|B'\|_2+\|B\|_2\big)\|B'-B\|_2$, hence so do their minima over that set,
namely $\alpha(B',\Sub)$ and $\alpha(B,\Sub)$.
\end{proof}

Part~\textbf{(i)} gives $\alpha(B,\Sub)>0$, so we may write
\be
\label{eq:kappa_def}
\kappa(B,\Sub):=\frac{\|B\|_2}{\sqrt{\alpha(B,\Sub)}}
\ee
for the associated restricted condition number, which satisfies
$1\le\kappa(B,\Sub)\le\lambda_{\max}(B)/\lambda_{\min}(B)$. The next two results
are proved in Appendix~\ref{app:perturbation}, and both rest on the real analyticity of
$A\mapsto\big(B^{*}(A),C^{*}(A)\big)$ on $\SPD$.

\begin{proposition}[Local stability of the decomposition]
\label{prop:perturbation_fro}
Let $A\in\SPD$ with decomposition $(B^{*},C^{*})$ given by
Theorem~\ref{thm:general_decomposition} for a nonzero subspace $\Sub$, and let
$\widetilde A=A+\Delta\in\SPD$ have decomposition $(\widetilde B,\widetilde C)$
for the same $\Sub$. Set
\be
\label{eq:smallness_fro}
\varepsilon:=\|B^{*}\|_2\,\kappa(B^{*},\Sub)^{2}\,\|\Delta\|_F
=\frac{\|B^{*}\|_2^{3}\,\|\Delta\|_F}{\alpha(B^{*},\Sub)} .
\ee
If $\varepsilon<1-\tfrac1{\sqrt2}$, then
\begin{align}
\label{eq:C_stab_fro}
\|\widetilde C-C^{*}\|_F
 &\le \frac{\kappa(B^{*},\Sub)}{\sqrt{1-4\varepsilon+2\varepsilon^{2}}}\,\|\Delta\|_F,\\
\label{eq:B_stab_fro}
\|\widetilde B-B^{*}\|_2 &\le
 \frac{\|B^{*}\|_2^{2}}{1-\|B^{*}\|_2\|\Delta\|_F}\,\|\Delta\|_F.
\end{align}
\end{proposition}

\begin{remark}
The bound~\eqref{eq:B_stab_fro} is free of $\alpha(B^{*},\Sub)$,
whereas~\eqref{eq:C_stab_fro} degrades with the restricted condition number
$\kappa(B^{*},\Sub)$ of~\eqref{eq:kappa_def}. Both inflation factors decrease
to $1$ as $\|\Delta\|_F\to0$, recovering the infinitesimal rates
$\kappa(B^{*},\Sub)$ and $\|B^{*}\|_2^{2}$
of~\eqref{eq:dC_bound} and~\eqref{eq:dB_bound}. Along any sequence with
$\alpha(B^{*},\Sub)\to0$,
Lemma~\ref{lem:alpha_lipschitz}\textbf{(i)} gives $\lambda_{\min}(B^{*})\to0$, hence
$\lambda_{\max}(A-C^{*})=\lambda_{\max}\big((B^{*})^{-1}\big)\to\infty$.
\end{remark}

Integrating the local bound along segments makes it uniform on compact
convex regions.

\begin{corollary}[Uniform stability on compact sets]
\label{cor:perturbation_global}
Let $\Sub$ be nonzero and let $\Theta\subset\SPD$ be nonempty, compact and convex,
and set
\[
\beta_0:=\max_{A\in\Theta}\|B^{*}(A)\|_2,
\qquad
\alpha_0:=\min_{A\in\Theta}\alpha\big(B^{*}(A),\Sub\big).
\]
Then, for all $A,\widetilde A\in\Theta$,
\[
\|C^{*}(\widetilde A)-C^{*}(A)\|_F\le\frac{\beta_0}{\sqrt{\alpha_0}}\,\|\widetilde A-A\|_F,
\qquad
\|B^{*}(\widetilde A)-B^{*}(A)\|_2\le \beta_0^{2}\,\|\widetilde A-A\|_F .
\]
\end{corollary}



\section{Algorithms and Complexity}
\label{sec:algorithms}
\label{sec:optimization}

This section presents numerical methods for computing the unique pair
$(B^\ast,C^\ast)\in\SPD\times \Sub$ from
Theorem~\ref{thm:general_decomposition}.
Writing
$C(x)=\sum_{k=1}^m x_k D_k$
and
$\Phi(x):=-\log\det\bigl(A-C(x)\bigr)$,
the decomposition is obtained by solving the convex optimization problem
\be
x^\ast \;\in\; \argmin_{x\in\Omega_x}\;\Phi(x),
\ee
and then setting $C^\ast=C(x^\ast)$ and $B^\ast=M(x^\ast)^{-1}$.

This formulation yields a smooth, strictly convex optimization problem over an
open convex domain.
Our focus here is not on advocating a particular optimization algorithm, but on
understanding how \emph{algebraic structure} in the data matrix $A$ and in the
constraint subspace $\Sub$ affects the computational complexity of solving this
problem.

Throughout, we use a \emph{Newton-CG} method as the default algorithm.
This choice avoids explicit formation of the Hessian and replaces the Newton
system solve by a matrix-free conjugate-gradient iteration.
Exact Newton methods, based on forming and factorizing the full Hessian, are
used only when $m=\dim \Sub$ is sufficiently small that the associated $m^3$ cost
is negligible.
The Newton-CG method is particularly convenient here because all derivatives
of $\Phi$ admit closed-form expressions in terms of linear solves with
$M(x)=A-C(x)$, allowing the impact of algebraic structure (such as sparsity,
bandedness, Toeplitz structure, or symmetry) to be tracked explicitly.

A detailed demonstration of our implementations on illustrative structural
examples is available in Appendix \ref{subsec:algebraic_examples}.
A reference Python implementation of all algorithms described in this section,
including structure-exploiting variants, is freely available at
\href{https://github.com/orzuk/ConstrainedDecomposition}{\emph{github.com/orzuk/ConstrainedDecomposition}}.

\subsection{Derivatives}
\label{subsec:derivatives}

Recall the parametrization $M(x)=A-C(x)$ and $B(x)=M(x)^{-1}$, and the gradient
elements $\frac{\partial\Phi}{\partial x_k}(x)=\tr\big(B(x)D_k\big)$ given in
\eqref{eq:grad_phi}. Differentiating once more yields the Hessian
\be\label{eq:hess_section2}
\frac{\partial^2\Phi}{\partial x_k\partial x_\ell}(x)
=
\tr\big(B(x)D_kB(x)D_\ell\big),
\qquad k,\ell=1,\dots,m.
\ee

Writing $D(v):=\sum_{\ell=1}^m v_\ell D_\ell$ for $v\in\R^m$, the Hessian
$H(x):=\nabla^2\Phi(x)$ satisfies
\[
v^\top H(x)v=\tr\big(B(x)D(v)B(x)D(v)\big)
=\big\|B(x)^{1/2}D(v)B(x)^{1/2}\big\|_F^2 ,
\]
which is strictly positive whenever $v\neq0$, since the basis matrices $D_k$ are
linearly independent and hence $D(v)\neq0$. Thus $H(x)\succ0$ for all
$x\in\Omega_x$, and $\Phi$ is strictly convex on $\Omega_x$.

All derivatives of $\Phi$ are computed using explicit analytic expressions and
require only applications of $M(x)^{-1}$, implemented through a factorization
or structure-exploiting solver for $M(x)$. No numerical differentiation is used.
In the structured settings the inverse $B(x)$ is never formed, since only the
entries of $M(x)^{-1}$ on the supports of the $D_k$ are needed. In the dense
case it is formed once per outer iteration, as accounted for in
Section~\ref{subsec:complexity_brief}. If
$B^\ast$ is required explicitly it is formed once on exit, at cost $O(n^3)$.
The Newton-CG method requires only Hessian-vector products. With $D(v)$ as above,
\be\label{eq:Hv_formula}
\big(H(x)v\big)_k
=
\tr\big(B(x)D_kB(x)D(v)\big),
\qquad k=1,\dots,m,
\ee
which can be evaluated using the same linear solves with $M(x)$. This matrix-free
access to second-order information is analogous in spirit to reverse-mode
automatic differentiation, but here all derivatives are available in closed
form.

\subsection{Newton-CG with feasibility-preserving backtracking}
\label{subsec:newton_algorithm}

The minimizer of $\Phi$ is computed by a Newton-CG method with
feasibility-preserving Armijo backtracking.

At a current iterate $x\in\Omega_x$, let $g=\nabla\Phi(x)$ and $H=\nabla^2\Phi(x)$.
The Newton direction $d$ is defined implicitly as the solution of
\be
H(x)d=-g(x),
\ee
which is computed using conjugate gradients with Hessian-vector products
$p\mapsto H(x)p$ given by \eqref{eq:Hv_formula}.
Step sizes are chosen by backtracking to ensure $M(x+td)\succ0$ and sufficient decrease of $\Phi$.

Algorithmic details, including evaluation of $\Phi$, its gradient, and
Hessian-vector products, are given in Appendix~\ref{app:newtoncg_pseudocode}.

\paragraph{Convergence}
The objective $\Phi(x)=-\log\det\bigl(A-C(x)\bigr)$ is strictly convex and self-concordant on
$\Omega_x$, as it is the composition of the self-concordant log-determinant barrier with an affine map.
By Theorem~\ref{thm:general_decomposition} and Lemma~\ref{lem:bounded_intersection},
the sublevel set $\{x:\Phi(x)\le\Phi(x_0)\}$ of any feasible $x_0$ is nonempty, compact, and bounded away from $\partial\Omega_x$.
Consequently, Newton-CG with feasibility-preserving Armijo backtracking is
globally convergent from any feasible initialization, provided each CG solve is
terminated by an inexact-Newton forcing condition
\be
\label{eq:forcing}
\|H(x_k)d_k+g(x_k)\|_2\le\eta_k\,\|g(x_k)\|_2,
\qquad \eta_k\le\bar\eta<1 .
\ee
Started from $d=0$, every CG iterate is a descent direction, since
$H(x_k)\succ0$.
The \emph{local} rate is governed by the forcing sequence $\{\eta_k\}$ rather
than by a fixed tolerance: $\eta_k\le\bar\eta<1$ gives a linear rate,
$\eta_k\to0$ a superlinear one, and $\eta_k=O(\|g(x_k)\|_2)$ a quadratic one, in
which case full Newton steps are eventually accepted near the unique minimizer
$x^\ast$. See, e.g.,
\cite{dembo1982inexact,nesterov1994interior,nocedal2006numerical}.

\subsection{Per-iteration cost and algorithmic considerations}
\label{subsec:complexity_brief}

Table~\ref{tab:per_iteration_complexity} reports the per-iteration cost of
Newton-type methods in terms of the ambient dimension $n$, the constraint
dimension $m=\dim \Sub$, and the parameters of whatever algebraic structure is
present. The exact Newton method appears only as a baseline for very small $m$. All entries
are per outer iteration and asymptotic, ignoring constants and lower-order
terms, and some are deliberately loose, being stated without assumptions on how
the basis $\{D_k\}$ is distributed within the admissible sparsity pattern, so
measured runtimes may scale better. The dominant operations are
(i) construction of a solver for $M(x)=A-C(x)$,
(ii) evaluation of the gradient traces $\tr\big(M(x)^{-1}D_k\big)$,
and (iii) computation of a second-order search direction by CG.

The dense bounds assume unstructured basis matrices $D_k$, so each trace evaluation costs $O(n^2)$.
Each conjugate-gradient step forms a Hessian--vector product
$H(x)p$ via~\eqref{eq:Hv_formula}, which requires applying $M(x)^{-1}$ to the
dense combination $D(p)=\sum_\ell p_\ell D_\ell$. In the dense case this is a
$O(n^3)$ solve (with $M(x)$ factored once per outer iteration), in addition
to $O(mn^2)$ for the $m$ trace evaluations; hence a Newton-CG outer
iteration costs $O\big(n^3+N_{\mathrm{cg}}(n^3+mn^2)\big)$. Precomputing the
$m$ matrices $M(x)^{-1}D_k$ instead costs $O(mn^3)$ once, with $O(mn^2)$
storage, and reduces each CG step to $O(mn^2)$. This pays only when the number
of CG steps is large relative to $m$. The structured rows of
Table~\ref{tab:per_iteration_complexity} follow the same accounting, with the
dense solve cost $n^3$ replaced by the corresponding structured solve cost.

In the banded setting, both $A$ and the constraint subspace $\Sub$ are chosen so
that $M(x)$ remains banded along the optimization path, with half-bandwidth $b$,
meaning that entries vanish for $|i-j|>b$. Linear solves and trace evaluations
are then carried out in $O(nb^2)$ and $O(b^2)$ time,
respectively. Although $M(x)^{-1}$ is dense in general, the trace terms involve only entries of $M(x)^{-1}$ within the sparsity pattern of the banded matrices $D_k$, which can be
computed efficiently using sparse inverse subset methods (e.g., Takahashi-type
recursions, see \cite{rue2005gaussian}).
In the banded row, the term $mb^2$ bounds the $m$ trace evaluations when each
$D_k$ is supported on a $b\times b$ block. For the standard banded basis, in
which each $D_k$ has $O(1)$ nonzero entries, this term reduces to $O(m)$.

\begin{remark}[Banded Hessian--vector products]
\label{rem:banded_hessvec}
By \eqref{eq:Hv_formula}, a conjugate-gradient step requires
\begin{align*}
\big(H(x)v\big)_k
&=\tr\big(M(x)^{-1}D_k\,M(x)^{-1}D(v)\big)
=\big\langle D_k,\,G\big\rangle_F,\\
G&:=M(x)^{-1}D(v)\,M(x)^{-1},
\end{align*}
so only the entries of $G$ on the supports of the $D_k$ are needed, that is, the
band of $G$. Unlike the gradient traces, this does not reduce to a single
structured solve: the matrix $M(x)^{-1}$ is dense and occurs \emph{twice} in
$G$, so applying it successively to $D(v)$ produces a dense intermediate and
costs $O(n^2)$ or more per product. The band of $G$ is nevertheless
computable in $O(nb^2)$. Indeed,
\be
\label{eq:G_as_derivative}
G=-\frac{d}{dt}\Big|_{t=0}\big(M(x)+t\,D(v)\big)^{-1},
\ee
so the required entries of $G$ form the directional derivative of the selected
inverse of $M(x)$. Equivalently, $G$ is the off-diagonal block of the inverse of the $2n\times2n$
matrix
\[
\begin{pmatrix} M(x) & -D(v)\\ 0 & M(x)\end{pmatrix},
\]
which is banded once the two copies are interleaved, index $i$ of the first
placed next to index $i$ of the second. In the ordering displayed above the
off-diagonal block instead sits $n$ positions away from the diagonal.
Grouping the indices into $n/b$ consecutive blocks of size $b$ makes $M(x)$
block tridiagonal. The block $LDL^\top$ factorization and the block
Takahashi recursion then each consist of $n/b$ steps of $O(b^3)$ work, and
propagating forward-mode derivatives through both recursions costs the same.
The band of $G$ obtained this way is exact. This is a manifestation of the
quasiseparable structure of the inverse of a banded matrix
\cite{vandebril2008matrix,erisman1975computing}: the mechanism that makes the
gradient traces cheap also makes the second-order terms cheap, and it is what
the $nb^2$ entries of Table~\ref{tab:per_iteration_complexity} rest on.
\end{remark}

In the circulant setting, $A$ and the basis matrices $\{D_k\}$ are circulant and
symmetric, so that $M(x)$ remains circulant along the optimization path and is
diagonalized by the discrete Fourier transform: writing $F$ for the unitary
discrete Fourier matrix, a symmetric circulant matrix has the form
$F^{-1}\diag(\cdot)F$, the diagonal holding its eigenvalues. Let
$\mu,\delta,d_k\in\R^n$ be the eigenvalue vectors of $M(x)$, $D(v)$ and $D_k$,
so that $M(x)=F^{-1}\diag(\mu)F$, $D(v)=F^{-1}\diag(\delta)F$ and
$D_k=F^{-1}\diag(d_k)F$. The two-sided product is then again circulant, and
\eqref{eq:Hv_formula} becomes a pointwise product of these vectors,
\begin{align}
M(x)^{-1}D(v)M(x)^{-1}&=F^{-1}\diag(\delta/\mu^{2})F ,\nonumber\\
\big(H(x)v\big)_k&=\sum_{j}(d_k)_j\,\delta_j/\mu_j^{2} .
\label{eq:circulant_G}
\end{align}
Two fast Fourier transforms
produce $\mu$ and $\delta$ in $O(n\log n)$, after which all $m$ trace
evaluations amount to a single $m\times n$ matrix--vector product. This is the
one structured case in which the second-order terms are exactly diagonalized.
See, e.g., \cite{gray2006toeplitz,gohberg1994fast}.

For a general Toeplitz $M(x)$, linear solves remain fast, requiring $O(n^2)$
time by Levinson/Schur-type recursions. The second-order terms, however, lose
the structure, since the inverse of a Toeplitz matrix is again Toeplitz only in
special cases, so at this generality $M(x)^{-1}D(v)M(x)^{-1}$ retains the dense
bound.

The row labeled ``$\calG$-invariant, $r$ blocks'' records the favorable case in which
$\calG$ consists of permutations of $r$ blocks and the symmetry-adapted transform is
a block averaging, so that the symmetry collapses the $n\times n$ linear algebra
to $r\times r$ and the whole iteration takes place in the $m_{\calG}$-dimensional
fixed-point subspace. The entry is derived in
Section~\ref{subsec:group_invariant_opt}.

More generally, scalability is governed by whether algebraic structure in $A$
and the constraint subspace $\Sub$ is preserved along the optimization path, as this
directly reduces the cost of linear solves with $M(x)$ and of trace evaluations.

\begin{table}[htbp]
\centering
\scriptsize
\renewcommand{\arraystretch}{1.15}
\setlength{\tabcolsep}{2pt}

\caption{Per-iteration worst-case complexity of Newton-type methods.
Here $m=\dim \Sub$, $m^\perp=\dim(\Sub^\perp)$, and $m_{\calG}=\dim(\Sub^{\calG})$ denote the effective
optimization dimensions in the primal, dual, and group-invariant formulations,
respectively. These quantities are defined and discussed in
Sections~\ref{subsec:dual_newton} and~\ref{subsec:group_invariant_opt}.
The parameter $b$ denotes half-bandwidth, $r$ the number of blocks in the
block-permutation case, $N_{\mathrm{cg}}$ the number of CG iterations, and
$T_{\mathrm{fact}}(\calG)$ the cost of the symmetry-adapted change of basis,
incurred once before the iteration begins.}
\label{tab:per_iteration_complexity}

\begin{tabular}{llll}
\toprule
Structure & Method & Per-iteration cost & Notes \\
\midrule
Dense & Newton-CG &
$O\big(n^3 + N_{\mathrm{cg}}\,(n^3 + m n^2)\big)$ &
default \\

Dense & Exact Newton &
$O(m n^3 + m^2 n^2 + m^3)$ &
small $m$ \\

Dense & Dual Newton-CG &
$O\big(n^3 + N_{\mathrm{cg}}\,(n^3 + m^\perp n^2)\big)$ &
$m^\perp<m$ \\

Banded ($b$) & Newton-CG &
$O\big(n b^2 + N_{\mathrm{cg}}\,(n b^2 + m b^2)\big)$ &
bandwidth preserved \\

Circulant & Newton-CG &
$O\big(n\log n + N_{\mathrm{cg}}\,(n\log n + m n)\big)$ &
FFT diagonalization \\

$\calG$-invariant & Newton-CG &
$O\big(n^3 + N_{\mathrm{cg}}\,(n^3 + m_{\calG} n^2)\big)$ &
plus one-time $T_{\mathrm{fact}}(\calG)$ \\

$\calG$-invariant, $r$ blocks & Newton-CG &
$O\big(r^3 + N_{\mathrm{cg}}(r^3 + m_{\calG} r^2)\big)$ &
independent of $n$ \\
\bottomrule
\end{tabular}
\end{table}

\subsection{Newton-CG for the dual problem}
\label{subsec:dual_newton}

As shown in Section~\ref{subsec:duality_general}, the decomposition
\eqref{eq:maindecomposition} can equivalently be obtained by solving the dual problem
\be
\max_{B\succ0,\; B\in \Sub^\perp}\;\; \log\det(B)-\tr(AB)+n.
\label{eq:dual_problem_section2}
\ee
This problem is smooth and strictly concave on the open convex set
$\Sub^\perp\cap\SPD$, and admits a unique maximizer $B^\ast$. As in the primal
case, we use a Newton-CG method as the default algorithm, and exact Newton
methods only when $\dim(\Sub^\perp)$ is sufficiently small.
Let $\{E_1,\dots,E_{m^\perp}\}$ be a basis of $\Sub^\perp$, where
\be
m^\perp := \dim(\Sub^\perp)=\tfrac{n(n+1)}2-m.
\ee
Fix any $B_0\in \Sub^\perp\cap\SPD$ and write
\be
B(y) := B_0 + \sum_{k=1}^{m^\perp} y_k E_k,
\qquad y\in\R^{m^\perp}.
\ee
Unlike the primal formulation, which is always initialized at $x=0$ (feasible
since $A\succ0$), the dual requires such a feasible $B_0$. Obtaining one is, in
general, a feasibility (phase-one) step: a semidefinite program in the same
$m^\perp$ variables as~\eqref{eq:dual_problem_section2}, whose solvability is
guaranteed by Theorem~\ref{thm:general_decomposition}. In the structured cases
of interest this step is immediate: whenever every matrix in $\Sub$ has zero
diagonal, as for zero-diagonal and off-diagonal prescribed-support constraints,
$\tr(I_nX)=\tr(X)=0$ for all $X\in \Sub$, so $I_n\in \Sub^\perp$ and one may take
$B_0=I_n$, with analogous explicit choices in the banded and group-invariant
settings.
Define
\be
\Psi(y) := \log\det\bigl(B(y)\bigr)-\tr\bigl(AB(y)\bigr)+n.
\ee
Then
\begin{align}
\frac{\partial\Psi}{\partial y_k}(y)
&=
\tr\big(B(y)^{-1}E_k\big)-\tr(AE_k), \nonumber \\
\frac{\partial^2\Psi}{\partial y_k\partial y_\ell}(y)
&=
-\tr\big(B(y)^{-1}E_kB(y)^{-1}E_\ell\big),
\qquad k,\ell=1,\dots,m^\perp .
\end{align}

\subsection{Group-invariant optimization}
\label{subsec:group_invariant_opt}

Assume that a finite group $\calG\subset\Orth(n)$ acts on $\Sym$ by conjugation
$X\mapsto P X P^\top$, and that both $A$ and $\Sub$ are $\calG$-invariant.
Define
\be
\Sub^{\calG} := \{X\in \Sub:\;P X P^\top=X\ \forall P\in\calG\},
\qquad
m_{\calG} := \dim \Sub^{\calG} \le m.
\ee
By \eqref{eq:G_fixed}, the unique minimizer
$C^\ast$ lies in $\Sub^{\calG}$. Writing
\be
C(x)=\sum_{k=1}^{m_{\calG}} x_k \widetilde D_k,
\ee
where $\{\widetilde D_k\}$ is a basis of $\Sub^{\calG}$, the optimization problem reduces
to dimension $m_{\calG}$.
Newton-CG applies verbatim with the reduced derivatives.

\paragraph{Complexity}
The per-iteration cost is that of Table~\ref{tab:per_iteration_complexity} with
$m$ replaced by $m_{\calG}$. In the generic dense case,
\be
\label{eq:group_newton_cost}
T_{\mathrm{Newton}}^{\calG}
=
O\big(n^3+N_{\mathrm{cg}}(n^3+m_{\calG} n^2)\big).
\ee
Reducing the optimization dimension from $m$ to $m_{\calG}$ shrinks the number of
trace evaluations, but not the $n\times n$ linear algebra: $M(x)$ is still an
$n\times n$ matrix that must be factorized once per outer iteration, and each
Hessian--vector product still applies $M(x)^{-1}$ twice. To this one adds the
cost $T_{\mathrm{fact}}(\calG)$ of the symmetry-adapted change of basis. That
transform depends only on $\calG$ and not on the iterate, so it is computed once
before the iteration begins, outside the per-iteration cost.

Further reduction occurs when the symmetry shrinks $M(x)$ itself. In the
block-permutation setting of the last row of
Table~\ref{tab:per_iteration_complexity}, Example~IV of
Appendix~\ref{app:examples}, $\calG$ permutes $r$ blocks and every
$\calG$-invariant matrix is determined by $O(r^2)$ coefficients, one for each
off-diagonal block pair and two for each diagonal block. The symmetry-adapted transform is a block averaging, so
$T_{\mathrm{fact}}(\calG)$ is one $O(n^2)$ pass over $A$, and nothing when $A$ is
supplied by its block coefficients. Solves with $M(x)$ and the two-sided
products of Remark~\ref{rem:banded_hessvec} then act on $r\times r$ data at cost
$O(r^3)$, and each of the $m_{\calG}$ trace evaluations costs $O(r^2)$, giving a
per-iteration complexity independent of $n$:
\be
\label{eq:group_block_cost}
T_{\mathrm{Newton}}^{\calG,\,r\text{ blocks}}
=
O\big(r^3+N_{\mathrm{cg}}(r^3+m_{\calG} r^2)\big),
\qquad m_{\calG}\le r(r+1)/2 .
\ee

\begin{remark}
For general finite groups, $T_{\mathrm{fact}}(\calG)$ may be comparable to dense
linear algebra.
For certain groups, special algebraic structure yields fast transforms (e.g.
FFTs for cyclic groups), dramatically reducing $T_{\mathrm{fact}}(\calG)$.
\end{remark}

\section{Exponential Utility Maximization}
 \label{sec:finance}
The utility maximization problem is central to mathematical finance, providing a rigorous framework for optimal decision-making under uncertainty and for linking investors’ preferences to market dynamics. A particularly important specification is exponential utility, which exhibits constant absolute risk aversion.
For background and details see \cite{merton1969lifetime}, Chapter~3 of \cite{karatzas1998methods}, \cite{schachermayer2004utility}, Chapter~2 of \cite{follmer2016stochastic}
and Chapter~10 in \cite{eberlein2019mathematical}.

We consider a discrete-time market model with normally distributed asset increments and a general information and hedging structure. Let \(N\in\mathbb{N}\) denote the time horizon. At each time \(i=1,\dots,N\), suppose that \(p_i\in\mathbb{N}\) risky assets are available for trading, with increments
$
X^i_1,\dots,X^i_{p_i}.
$
In addition, we allow for further random variables
$
X^i_{p_i+1},\dots,X^i_{q_i}$,
$q_i\ge p_i,$
which represent increments of other risky factors or assets that may influence the market but are not tradable at time \(i\).
The information available for trading at time \(i\) is generated by the random vector
\[
Y^i=(Y^i_1,\dots,Y^i_{m_i})
=
\bigl(
X^1_1,\dots,X^1_{q_1},\;
\dots,\;
X^{i-1}_1,\dots,X^{i-1}_{q_{i-1}}
\bigr) A^i,
\]
where \(m_i \le \sum_{j=1}^{i-1} q_j\) and \(A^i\) is a deterministic matrix. Thus, the investor’s information at time \(i\) consists of linear combinations of past market increments.

Let \(n=\sum_{i=1}^N q_i\), and assume that the random vector
\[
Z=(Z_1,...,Z_n):=\bigl(X^1_1,\dots,X^1_{q_1},\;\dots,\;X^N_1,\dots,X^N_{q_N}\bigr)
\]
has a non-degenerate multivariate normal distribution with mean zero (for simplicity) and covariance matrix \(\Sigma\in\SPD\). Denote by \(\pre=\Sigma^{-1}\) the corresponding precision matrix. 

A trading strategy is a random vector
$
\gamma=\bigl(
\gamma^1_1,\dots,\gamma^1_{p_1},\;
\dots,\;
\gamma^N_1,\dots,\gamma^N_{p_N}
\bigr),
$
where \(\gamma^i_j\) denotes the number of shares held in the risky asset with increment \(X^i_j\). The position \(\gamma^i_j\) is required to be \(Y^i\)-measurable, that is,
\be\label{strategy}
\gamma^i_j = g^i_j(Y^i) \; \text{for a measurable}\ 
g^i_j:\mathbb{R}^{m_i}\to\mathbb{R},
\, i=1,\dots,N,\ j=1,\dots,p_i.
\ee
The corresponding terminal portfolio value is
\begin{equation}\label{portfolio}
V^\gamma
=
\sum_{i=1}^N\sum_{j=1}^{p_i}\gamma^i_jX^i_j.
\end{equation}
Each term \(\gamma^i_jX^i_j\) represents the profit (or loss) generated by holding \(\gamma^i_j\) units of the \(j\)-th tradable asset over period \(i\). Hence, \(V^\gamma\) is the total cumulative trading gain over the investment horizon. We denote by \(\mathcal A\) the set of all trading strategies \(\gamma\) for which \(V^\gamma\) is integrable under every centered nondegenerate Gaussian law of \(Z\), which is what makes the change-of-measure identities used below well defined.

Next, we introduce some notation. 
Let \(\mathcal A_l \subset \mathcal A\) be the set of trading strategies corresponding to linear functions $g^i_j$ in~\eqref{strategy} ($g^i_j(0)=0$). For these $V^\gamma$ is a quadratic form in $Z$ and the integrability requirement holds automatically.
For any \(i=1,\ldots,N\), \(j=1,\ldots,p_i\), and \(k=1,\ldots,m_i\), the product \(Y^i_k X^i_j\) can be represented as,
\be\label{representation}
Y^i_k X^i_j = -\frac{1}{2}\, Z D^{i,j,k} Z^\top,
\ee
where \(D^{i,j,k}\in\Sym\) is the unique symmetric matrix representing this quadratic form. Since the random vector $Y^i$ is 
determined by 
$X^r_j$, $r<i$, $j=1,\ldots,q_r$, we obtain that $D^{i,j,k}$
is a zero-diagonal matrix.
Let \(\band\subset\Sym\) denote the linear span of all matrices of the form \(D^{i,j,k}\).

A strategy \(\gamma\) lies in \(\mathcal A_l\) if and only if there exists
$X\in\band$ such that
\be\label{eq:v_gamma_bilinear}
V^{\gamma} = -\frac{1}{2} Z X Z^\top .
\ee

The investor’s preferences are described by the exponential utility function
\[
u(x)=-\exp(-\rho x), \qquad x\in\mathbb{R},
\]
with absolute risk aversion parameter \(\rho>0\). The optimization problem is therefore
\be\label{eq:exp_util}
\max_{\gamma\in\mathcal{A}}
\ \mathbb{E}_{\mathbb{P}}\!\left[-\exp\!\left(-\rho V^\gamma\right)\right],
\ee
where \(\mathbb{P}\) is the probability measure under which the Gaussian vectors
\(X^i_j\) and \(Y^i\) introduced above are distributed, and
\(\mathbb{E}_{\mathbb{P}}\) denotes expectation with respect to it.
For any $\gamma\in\mathcal A$ and \(\lambda\in\mathbb{R}\) we have 
$\lambda\gamma\in\mathcal A$ and 
\(V^{\lambda\gamma}=\lambda V^\gamma\). Hence, without loss of generality, we set \(\rho:=1\).

By \eqref{eq:v_gamma_bilinear}, a linear trading strategy corresponds to a matrix
\(X\in\band\) with terminal value \(-\frac12 ZXZ^\top\), so that \(\band\) is the
matrix counterpart of the set of admissible linear strategies. The exponential
utility maximization problem can therefore be viewed as a Gaussian
log-determinant optimization problem over the subspace \(\band\).

The following solves the exponential utility maximization problem \eqref{eq:exp_util}.
\begin{proposition}\label{thm:finance_value}
There exists a unique decomposition 
\be\label{eq:finance_decomposition}
\pre=\hat Q^{-1}+\hat\Gamma, \ \ \hat Q\succ 0, \ \ \hat Q\in \band^{\perp}, \ \ \hat\Gamma\in\band.
\ee
The maximizer $\hat\gamma$ for the optimization problem \eqref{eq:exp_util} is unique (and lies in $\mathcal A_l$) with terminal wealth given by 
\be\label{eq:optimal_v_gamma_bilinear}
V^{\hat\gamma}=-\frac{1}{2} Z \hat\Gamma Z^{\top}. 
\ee
The corresponding value is given by 
\be\label{eq:expected_value}
\mathbb E_{\mathbb P}\left[-\exp\left(- V^{\hat\gamma}\right)\right]=-\sqrt{\frac{\det(\hat Q)}{\det(\Sigma)}}.
\ee
\end{proposition}

\begin{remark}
The framework of Proposition \ref{thm:finance_value} shows that, in the Gaussian setting with exponential utility, 
the optimal portfolio is linear and can be characterized through a specific matrix decomposition.
The model accommodates partial information, non-traded assets, the trading of indices rather than individual securities, and an information structure that need not be a filtration, so that information may be lost over time.
\end{remark}

\begin{proof}[Proof of Proposition~\ref{thm:finance_value}]
Every matrix in \(\band\) has zero diagonal. If \(X\in\band\cap\mathbb S_+^n\), then \(X\) is positive semidefinite and \(X_{ii}=0\) for every \(i\). For a positive semidefinite matrix,
$
|X_{ij}|^2\le X_{ii}X_{jj}=0$,
and therefore \(X_{ij}=0\) for every \(i,j\). Hence \(X=0\), and consequently
$
\band\cap\mathbb S_+^n=\{0\}.
$
Thus, condition~\eqref{eq:nopsddirections} holds, and Theorem~\ref{thm:general_decomposition} yields the existence and uniqueness of the decomposition~\eqref{eq:finance_decomposition}.

Define a probability measure $\hat{\mathbb Q}$ such that
the distribution of $(Z_1,...,Z_n)$ is multivariate normal with mean zero and covariance matrix $\hat Q$. From \eqref{representation} 
it follows that for any \(i=1,\ldots,N\), \(j=1,\ldots,p_i\), and \(k=1,\ldots,m_i\)
\[
\mathbb E_{\hat{\mathbb Q}}\left[Y^i_k X^i_j\right] = -\frac{1}{2}\mathbb E_{\hat{\mathbb Q}}\left[Z D^{i,j,k} Z^\top\right]=-\frac{1}{2}\tr\left(D^{i,j,k} \hat Q\right)=0.
\]
Hence, since $Z$ is Gaussian and the above cross-covariances vanish, the random vectors $(X^i_1,\dots,X^i_{p_i})$ and $Y^i$ are independent under $\hat{\mathbb Q}$  for any $i\leq N$.
Thus, from the tower property of conditional expectation we conclude 
\be\label{nuetral}
\mathbb E_{\hat{\mathbb Q}}\left[V^{\gamma}\right]=0, \ \ \forall \gamma\in\mathcal A.
\ee

Next, choose $\gamma\in\mathcal A$.
For any $z>0$
\begin{align*}
\mathbb E_{\mathbb P}\left[\exp\left(-V^{ \gamma}\right)\right]
&=\mathbb E_{\mathbb P}\left[\exp\left(-V^{ \gamma}\right)+z \frac{d{\hat{\mathbb Q}}}{d\mathbb P}V^{\gamma}\right]\\
&\geq \mathbb E_{\mathbb P}\left[z\frac{d{\hat{\mathbb Q}}}{d\mathbb P}\left(1-\log
\left(z\frac{d{\hat{\mathbb Q}}}{d\mathbb P}\right)\right)\right]\\
&=z-z\log z-z\mathbb E_{\hat{\mathbb Q}}\left[\log
\left(\frac{d{\hat{\mathbb Q}}}{d\mathbb P}\right)\right].
\end{align*}
The first equality is due to \eqref{nuetral}.
The above inequality follows from the well-known relation $xy \leq e^x+y(\log y-1)$ for all $x \in \mathbb{R}, y>0$ by setting $x=-V^{\gamma}$ and 
$y=z\frac{d\hat{\mathbb Q}}{d\mathbb P}$, and rearranging terms.
The last equality is straightforward. 
From simple calculus it follows that the concave function defined for $z>0$ by
$$z\mapsto z-z\log z-z\,\mathbb E_{\hat{\mathbb Q}}\left[\log
\left(\frac{d{\hat{\mathbb Q}}}{d\mathbb P}\right)\right]$$
attains its maximum at $z^{*}:=\exp\left(-\mathbb E_{\hat{\mathbb Q}}\left[\log
\left(\frac{d{\hat{\mathbb Q}}}{d\mathbb P}\right)\right]\right)$ and the corresponding maximal value
is also $z^{*}$. Thus,
\be\label{lowerbound}
\mathbb E_{\mathbb P}\left[\exp\left(-V^{ \gamma}\right)\right]\geq \exp\left(-\mathbb E_{\hat{\mathbb Q}}\left[\log
\left(\frac{d{\hat{\mathbb Q}}}{d\mathbb P}\right)\right]\right), \ \ \forall \gamma\in\mathcal A.
\ee
The bound \eqref{lowerbound} was derived for an arbitrary
$\gamma\in\mathcal A$, with no restriction to $\mathcal A_l$. We now exhibit a
\emph{candidate} strategy attaining it. Since $\hat\Gamma\in\band$,
the characterization \eqref{eq:v_gamma_bilinear} supplies a strategy
$\hat\gamma\in\mathcal A_l$ satisfying \eqref{eq:optimal_v_gamma_bilinear}. At
this stage $\hat\gamma$ is only a candidate, and its optimality is established
below.
For multivariate normal distributions the likelihood ratio is
\be
\frac{d{{\hat{\mathbb Q}}}}{d\mathbb P}=\frac{\exp\left(-\frac{1}{2}Z\hat Q^{-1}Z^\top \right)/\det(\hat Q)^{1/2}}{\exp\left(-\frac{1}{2}Z \pre Z^\top\right)/\det(\Sigma)^{1/2}}.
\label{eq:radon_nykodim}
\ee 
Thus, from \eqref{eq:finance_decomposition}, \eqref{eq:optimal_v_gamma_bilinear} and \eqref{eq:radon_nykodim} we obtain 
\[
V^{\hat\gamma}+\log\left(\frac{d{{\hat{\mathbb Q}}}}{d\mathbb P}\right)=\frac{1}{2}\left(\log\det(\Sigma)-\log\det(\hat Q)\right).
\]
Hence, from \eqref{nuetral}
\begin{align}
\mathbb E_{\mathbb P}\left[\exp\left(-V^{\hat \gamma}\right)\right]
&=\exp\left(-\frac{1}{2}\left(\log\det(\Sigma)-\log\det(\hat Q)\right)\right)
\nonumber\\
&=\exp\left(-\mathbb E_{{\hat{\mathbb Q}}}\left[\log
\left(\frac{d{\hat{\mathbb Q}}}{d\mathbb P}\right)\right]\right).
\label{equality}
\end{align}
Comparing \eqref{equality} with \eqref{lowerbound}, the candidate $\hat\gamma$
attains the bound that holds for every $\gamma\in\mathcal A$. Hence $\hat\gamma$
is an optimal portfolio, and the value is given by the right hand side of \eqref{eq:expected_value}.
For uniqueness, equality in $xy\leq e^x+y(\log y-1)$ holds only when $x=\log y$, so any optimal
$\gamma$ satisfies $-V^{\gamma}=\log\left(z^{*}\frac{d\hat{\mathbb Q}}{d\mathbb P}\right)$ almost surely.
The optimal terminal wealth is therefore unique. Finally, the map $\gamma\mapsto V^{\gamma}$ is injective:
conditionally on the increments before time $i$, the strategy $\gamma^i$ is deterministic and the covariance
of $(X^i_1,\ldots,X^i_{p_i})$ is nondegenerate, so $V^{\gamma}=V^{\gamma'}$ almost surely forces
$\gamma^i=\gamma'^i$, and backward induction on $i=N,\ldots,1$ gives $\gamma=\gamma'$. Hence the optimal
portfolio is unique.
\end{proof}
\section{A triangular linear algebra result}
\label{sec:linear_algebra}
Let $\calU_n\subset\calM_n(\R)$ denote the strictly upper triangular
matrices, and $\calL_n\subset\calM_n(\R)$ the lower triangular matrices
(with no restriction on the diagonal).

Given $A,\pre\in\SPD$, we seek a pair
$(U,L)\in\calU_n\times\calL_n$ satisfying
\begin{equation}
\label{eq:twisted_decomposition}
\pre=L+U+UAL,
\end{equation}
which we call the \emph{bilinear triangular decomposition} of $\pre$. The
pair $(A,\pre)$ is \emph{triangularly regular} when it exists. We present a direct algorithm that computes
the decomposition: its step count is fixed by $n$ alone and its output is exact
in exact arithmetic, in contrast to the iterative algorithms of
Section~\ref{sec:optimization}.
Theorem~\ref{thm:triangular_decomposition} proves the algorithm's correctness
and the uniqueness of the decomposition whenever it exists, while a vanishing
scalar pivot in the algorithm certifies non-existence.
Proposition~\ref{prop:variational_U} then relates the decomposition to
Theorem~\ref{thm:general_decomposition}: applied to a $2n\times2n$ matrix built
from $A$ and $\pre$, with $\Sub$ the subspace of off-diagonal blocks
$\bigl(\begin{smallmatrix}0&U\\U^\top&0\end{smallmatrix}\bigr)$ with
$U\in\calU_n$, the theorem produces a $U$ for every positive definite pair,
including those for which \eqref{eq:twisted_decomposition} has no solution, and
triangular regularity is exactly the invertibility of $I_n+UA$.

Algorithm~\ref{alg:backward_elimination} is the natural analogue of $LU$
decomposition without pivoting, with four differences: (i)~elimination proceeds \emph{backward} (from the
lower-right corner upward); (ii)~each step has one scalar pivot $\pi_k$
preceded by a triangular row-solve $M_0^{\top}u=\pre_{J,k}$ with
$M_0=I_{|J|}+A_{J,J}L_{J,J}$; (iii)~a zero pivot certifies non-existence of
the decomposition, on an algebraic set of positive definite pairs $(A,\pre)$;
and (iv)~the algorithm then halts, rather than recovering by a row swap as
partial pivoting does in standard $LU$.

At step $k$, with $J=\{k+1,\ldots,n\}$ and the trailing-block decomposition
$\pre_{J,J}=L_{J,J}+U_{J,J}+U_{J,J}A_{J,J}L_{J,J}$ already in hand, the
algorithm determines the new row $U_{k,J}$ of $U$, the new diagonal entry
$L_{kk}$, and the column $L_{J,k}$.

\begin{algorithm}[H]
\caption{Backward triangular elimination for $\pre=L+U+UAL$}
\label{alg:backward_elimination}
\begin{algorithmic}[1]
\Require $A,\pre\in\SPD$
\Ensure The unique pair $(U,L)$, or a zero-pivot certificate of nonexistence
\State $U\gets\mathbf{0}_{n,n}$, \; $L\gets\mathbf{0}_{n,n}$, \; $L_{nn}\gets\pre_{nn}$
\For{$k=n-1,n-2,\ldots,1$}
    \State $J\gets\{k+1,\ldots,n\}$
    \State $K_0\gets I_{|J|}+U_{J,J}A_{J,J}$, \; $M_0\gets I_{|J|}+A_{J,J}L_{J,J}$ \label{lin:form}
    \State Solve $M_0^\top u=\pre_{J,k}$, \; $U_{k,J}\gets u^\top$ \label{lin:solveM}
    \State Solve $K_0\,z=\pre_{J,k}$ and $K_0\,w=U_{J,J}\,A_{J,k}$ \label{lin:solveK}
    \State $\pi_k\gets 1+u^\top A_{J,k}-u^\top A_{J,J}\,w$
    \If{$\pi_k=0$}
        \State \textbf{stop:} \eqref{eq:twisted_decomposition} has no solution
    \EndIf
    \State $L_{kk}\gets(\pre_{kk}-u^\top A_{J,J}\,z)/\pi_k$, \; $L_{J,k}\gets z-w\,L_{kk}$
\EndFor
\State \Return $U,L$
\end{algorithmic}
\end{algorithm}

\begin{remark}[Efficient implementation, $O(n^3)$]
\label{rem:complexity}
Implemented naively, Algorithm~\ref{alg:backward_elimination} costs $O(n^4)$: lines~\ref{lin:form}--\ref{lin:solveK} build $M_0$ and $K_0$ afresh at
each step and solve with them, which is $O\bigl((n-k)^3\bigr)$ work. The overall cost
reduces to $O(n^3)$, matching standard LU, if those three lines are replaced by
an incremental update of the two inverses $M_0^{-1}$ and $K_0^{-1}$ as $k$
decreases, followed by three matrix--vector products. Between
consecutive steps the new index is prepended to the trailing block,
which writes
$M_0=\bigl(\begin{smallmatrix}\tau & y^\top\\ x & M_0^{\rm prev}\end{smallmatrix}\bigr)$
with $(M_0^{\rm prev})^{-1}$ already in hand. The Schur-complement
bordering identity
\[
M_0^{-1}=
\begin{pmatrix}
1/\sigma & -\hat y^\top/\sigma\\[2pt]
-\hat x/\sigma & (M_0^{\rm prev})^{-1}+\hat x\hat y^\top/\sigma
\end{pmatrix},
\]
with
\[
\hat x=(M_0^{\rm prev})^{-1}x,\qquad \hat y^\top=y^\top (M_0^{\rm prev})^{-1},\qquad \sigma=\tau-y^\top\hat x,
\]
updates the inverse in $O\bigl((n-k)^2\bigr)$, and the analogous formula updates
$K_0^{-1}$. Lines~\ref{lin:solveM} and~\ref{lin:solveK} then become the
matrix--vector products $u=(M_0^{\top})^{-1}\pre_{J,k}$,
$z=K_0^{-1}\pre_{J,k}$ and $w=K_0^{-1}U_{J,J}A_{J,k}$, each of cost
$O\bigl((n-k)^2\bigr)$. Summing over $k$ gives
$O(n^3)$.
\end{remark}

The correctness and uniqueness arguments rest on the following regularity
lemma, which rules out the degenerate possibility that a solution
of~\eqref{eq:twisted_decomposition} coexists with a singular $I_n+UA$.

\begin{lemma}[Regularity]
\label{lem:regularity}
Let $A,\pre\in\SPD$. If
$(U,L)\in\calU_n\times\calL_n$ satisfies \eqref{eq:twisted_decomposition}, then $I_n+UA$ is
invertible, and hence $L=(I_n+UA)^{-1}(\pre-U)$.
\end{lemma}
\begin{proof}
The identity is equivalent to $(I_n+UA)L=\pre-U$. Suppose $v\in\R^n$
satisfies $v^\top(I_n+UA)=0$. Then $v^\top(\pre-U)=v^\top(I_n+UA)L=0$, so
$v^\top\pre=v^\top U$. On the other hand $v^\top(I_n+UA)=0$ gives
$v^\top UA=-v^\top$, and multiplying by $A^{-1}$ on the right,
$v^\top U=-v^\top A^{-1}$. Combining the two identities yields
$v^\top(\pre+A^{-1})=0$. Since $\pre+A^{-1}$ is positive definite, $v=0$;
thus $I_n+UA$ is invertible.
\end{proof}

\begin{theorem}[Bilinear triangular decomposition]
\label{thm:triangular_decomposition}
Let $A,\pre\in\SPD$.
\begin{enumerate}
\item[\textbf{(i)}] Algorithm~\ref{alg:backward_elimination} is well defined:
at every step the matrices $M_0$ and $K_0$ are invertible.
\item[\textbf{(ii)}] Whenever a pair $(U,L)\in\calU_n\times\calL_n$ satisfying
\eqref{eq:twisted_decomposition} exists, it is unique,
$L=(I_n+UA)^{-1}(\pre-U)$, and $I_n+UA$ is invertible with
$\det(I_n+UA)=\prod_{k=1}^{n-1}\pi_k$. In addition, no pivot vanishes, so that
Algorithm~\ref{alg:backward_elimination} reaches its last step and returns this
pair.
\item[\textbf{(iii)}] Otherwise Algorithm~\ref{alg:backward_elimination} stops
at a zero pivot $\pi_k=0$, and this certifies that no pair in
$\calU_n\times\calL_n$ satisfies \eqref{eq:twisted_decomposition}.
\end{enumerate}
\end{theorem}

\begin{proof}
Fix $k$ and $J=\{k+1,\ldots,n\}$. Because any $U\in\calU_n$ is strictly upper
triangular and any $L\in\calL_n$ is lower triangular, the identity
\eqref{eq:twisted_decomposition} restricted to $\{k\}\cup J$ decouples into the trailing block
\begin{equation}
\label{eq:trailing_block}
\pre_{J,J}=L_{J,J}+U_{J,J}+U_{J,J}A_{J,J}L_{J,J},
\end{equation}
the upper-right block $\pre_{k,J}=U_{k,J}\,M_0$ with
$M_0=I_{|J|}+A_{J,J}L_{J,J}$, and the first-column block
\begin{equation}
\label{eq:col_block}
\begin{pmatrix}1+u^\top A_{J,k} & u^\top A_{J,J}\\ U_{J,J}\,A_{J,k} & K_0\end{pmatrix}
\begin{pmatrix}L_{kk}\\ L_{J,k}\end{pmatrix}
=\begin{pmatrix}\pre_{kk}\\ \pre_{J,k}\end{pmatrix},
\qquad u^\top:=U_{k,J},
\end{equation}
where $K_0=I_{|J|}+U_{J,J}A_{J,J}$.

\textbf{(i)} \emph{Well-posedness}. Whenever the trailing
relation~\eqref{eq:trailing_block} holds, Sylvester's determinant identity
$\det(I+XY)=\det(I+YX)$ gives
\begin{align}
\det(M_0)\,\det(K_0)
&=\det(I_{|J|}+A_{J,J}L_{J,J})\,\det(I_{|J|}+U_{J,J}A_{J,J})\nonumber\\
&=\det(I_{|J|}+A_{J,J}L_{J,J})\,\det(I_{|J|}+A_{J,J}U_{J,J})\nonumber\\
&=\det\!\big[(I_{|J|}+A_{J,J}U_{J,J})(I_{|J|}+A_{J,J}L_{J,J})\big]\nonumber\\
&=\det\!\big[I_{|J|}+A_{J,J}\big(L_{J,J}+U_{J,J}+U_{J,J}A_{J,J}L_{J,J}\big)\big]\nonumber\\
&=\det(I_{|J|}+A_{J,J}\pre_{J,J}) ,
\label{eq:mk_det}
\end{align}
the last step by~\eqref{eq:trailing_block}. This determinant is positive:
$A_{J,J}\pre_{J,J}$ is similar to
$A_{J,J}^{1/2}\pre_{J,J}A_{J,J}^{1/2}\succ0$, so its eigenvalues are positive
and those of $I_{|J|}+A_{J,J}\pre_{J,J}$ exceed $1$. Hence
$\det(M_0)\det(K_0)>0$, so neither factor vanishes and both $M_0$ and $K_0$ are
invertible. Since the algorithm
maintains~\eqref{eq:trailing_block} for the trailing block by construction,
every step is well defined.

\textbf{(ii)} \emph{Construction, uniqueness and the determinant}. We show by backward
induction on $k$ that the entries of \emph{any} pair
$(U,L)\in\calU_n\times\calL_n$ with \eqref{eq:twisted_decomposition} are forced, and equal those
computed by the algorithm. The base case is $L_{nn}=\pre_{nn}$, forced by the
$(n,n)$ entry. For the inductive step, suppose $U_{J,J},L_{J,J}$ are
determined. By~\eqref{eq:mk_det}, $M_0$ is invertible, so the upper-right block
gives $U_{k,J}=(M_0^\top)^{-1}\pre_{J,k}$. Taking the Schur complement of $K_0$
in~\eqref{eq:col_block} reduces it to $\pi_k\,L_{kk}=\pre_{kk}-u^\top A_{J,J}z$,
where $z=K_0^{-1}\pre_{J,k}$ and $\pi_k=1+u^\top A_{J,k}-u^\top A_{J,J}w$ with
$w=K_0^{-1}U_{J,J}A_{J,k}$. If $\pi_k\neq0$, then $L_{kk}$ and
$L_{J,k}=z-w\,L_{kk}$ are determined. Hence a solution, if it exists, is unique
and coincides with the algorithm's output. Conversely, if no zero pivot occurs
the algorithm produces a pair satisfying \eqref{eq:twisted_decomposition}.
For the determinant, taking determinants of the
block form of $I_{|K|}+U_{K,K}A_{K,K}$ for $K=\{k\}\cup J$ and using the Schur
complement of $K_0$ gives $\det(I_{|K|}+U_{K,K}A_{K,K})=\pi_k\,\det(K_0)
=\pi_k\,\det(I_{|J|}+U_{J,J}A_{J,J})$, so by induction
$\det(I_n+UA)=\prod_{k=1}^{n-1}\pi_k$. This proves~\textbf{(ii)}.

\textbf{(iii)} \emph{Non-existence}. Suppose the algorithm stops at a
zero pivot $\pi_k=0$ but a pair $(U,L)\in\calU_n\times\calL_n$ with
\eqref{eq:twisted_decomposition} existed. By the forcing above it agrees with the algorithm
through step $k$, so its $k$-th pivot is also zero and hence
$\det(I_n+UA)=\prod_j\pi_j=0$; but Lemma~\ref{lem:regularity} makes $I_n+UA$
invertible, a contradiction. Thus no solution exists.
\end{proof}

The next proposition connects the decomposition to the general result of
Section~\ref{sec:general_decomposition}: the matrix $U$ is exactly the
component produced by applying Theorem~\ref{thm:general_decomposition} to a
$2n\times2n$ matrix. This viewpoint also supplies the positive definiteness of
the matrix $S_U$, defined for $U\in\calU_n$ by
\be
\label{eq:S_U_def}
S_U:=\pre-U-U^\top-UAU^\top .
\ee

\begin{proposition}[Variational form of $U$; positivity of $S_U$]
\label{prop:variational_U}
Let $A,\pre\in\SPD$.
\begin{enumerate}
\item[\textbf{(i)}] There is a unique $U\in\calU_n$ such that $S_U\succ0$ and
$S_U^{-1}(I_n+UA)$ is lower triangular. It is obtained as
$C=\bigl(\begin{smallmatrix}0&U\\U^\top&0\end{smallmatrix}\bigr)$, the
$\calS$-component of the constrained decomposition of $B^{-1}$ furnished by
Theorem~\ref{thm:general_decomposition}, with $B$ as in~\eqref{eq:B_def}.
\item[\textbf{(ii)}] $(A,\pre)$ is triangularly regular if and only if
$I_n+UA$ is invertible for this $U$. In that case $U$ is the matrix of the
decomposition~\eqref{eq:twisted_decomposition}. In particular the
decomposition's $U$ satisfies $S_U\succ0$.
\end{enumerate}
\end{proposition}

\begin{proof}
\textbf{(i)} Set $\cov:=\pre^{-1}$ and consider the $2n\times 2n$ positive
definite matrix
\begin{equation}
\label{eq:B_def}
B:=
\begin{pmatrix}
\cov & \cov\\
\cov & \cov+A
\end{pmatrix},
\qquad
B^{-1}
=
\begin{pmatrix}
\pre+A^{-1} & -A^{-1}\\
-A^{-1} & A^{-1}
\end{pmatrix},
\end{equation}
together with the subspace
\[
\calS=
\left\{
\begin{pmatrix}
0 & U\\
U^\top & 0
\end{pmatrix}
:U\in\calU_n
\right\}\subset\calM_{2n}(\R)
\]
and its orthogonal complement (with respect to the trace inner product)
in the symmetric matrices,
\[
\calS^{\perp}
=
\left\{
\begin{pmatrix}
V & L\\
L^\top & W
\end{pmatrix}
:V,W \text{ symmetric},\, L\in\calL_n
\right\}.
\]
Every matrix in $\calS$ has zero diagonal, so as in the proof of
Proposition~\ref{thm:finance_value} a positive semidefinite $X\in\calS$ satisfies
$|X_{ij}|^2\le X_{ii}X_{jj}=0$ and hence $X=0$. Thus
$\calS\cap\mathbb S_+^{2n}=\{0\}$ and Theorem~\ref{thm:general_decomposition}
applies. It therefore supplies a unique pair
$(C,Q)\in\calS\times\calS^{\perp}$ such that $Q$ is positive definite
and $B^{-1}=Q^{-1}+C$. Writing
$C=\bigl(\begin{smallmatrix}0 & U\\ U^\top & 0\end{smallmatrix}\bigr)$, we compute
\[
B^{-1}-C
=
\begin{pmatrix}
\pre+A^{-1} & -A^{-1}-U\\
-A^{-1}-U^\top & A^{-1}
\end{pmatrix}.
\]
The Schur complement of the $(2,2)$-block $A^{-1}$ in $B^{-1}-C$ is
\begin{align*}
(\pre+A^{-1}) &- (-A^{-1}-U)\,A\,(-A^{-1}-U^\top)  \\
&=\pre+A^{-1}-(I_n+UA)(A^{-1}+U^\top)\\
&=\pre-U-U^\top-UAU^\top\;=\;S_U.
\end{align*}
The standard block-inverse formula then gives, for $Q=(B^{-1}-C)^{-1}$,
\[
Q_{11}=S_U^{-1},
\qquad
Q_{12}=-S_U^{-1}(-A^{-1}-U)\,A=S_U^{-1}(I_n+UA).
\]
Since $Q\in\calS^{\perp}$, the block $Q_{12}$ is lower triangular, and $Q\succ0$
gives $S_U=Q_{11}^{-1}\succ0$. This proves existence.

For uniqueness, let $U\in\calU_n$ be \emph{any} matrix with $S_U\succ0$ and
$S_U^{-1}(I_n+UA)$ lower triangular, and set
$C=\bigl(\begin{smallmatrix}0 & U\\ U^\top & 0\end{smallmatrix}\bigr)\in\calS$.
The computation above shows that $B^{-1}-C$ has $(2,2)$-block $A^{-1}\succ0$ and
Schur complement $S_U\succ0$, so $B^{-1}-C\succ0$. Set $Q:=(B^{-1}-C)^{-1}\succ0$.
Its blocks are again $Q_{11}=S_U^{-1}$ and $Q_{12}=S_U^{-1}(I_n+UA)$, the latter
lower triangular by hypothesis, so $Q\in\calS^{\perp}$. Hence $(C,Q)$ is a
constrained decomposition of $B^{-1}$, and the uniqueness in
Theorem~\ref{thm:general_decomposition} forces it to be the one constructed
above. Therefore $U$ is unique.

\textbf{(ii)} For any $(U',L')\in\calU_n\times\calL_n$ satisfying
$\pre=L'+U'+U'AL'$, a direct computation gives
\begin{equation}
\label{eq:S_U_factor}
S_{U'}=(I_n+U'A)(L'-U'^\top).
\end{equation}
($\Leftarrow$) If $I_n+UA$ is invertible for the $U$ of part~\textbf{(i)}, set
$L:=(I_n+UA)^{-1}(\pre-U)$. Then $L-U^\top=(I_n+UA)^{-1}S_U$, and since
$S_U^{-1}(I_n+UA)$ is lower triangular so is its inverse $L-U^\top$. Adding
$U^\top$ yields $L\in\calL_n$, so $(U,L)$
solves~\eqref{eq:twisted_decomposition} and $(A,\pre)$ is regular.

($\Rightarrow$) If $(A,\pre)$ is regular, then by
Theorem~\ref{thm:triangular_decomposition} it has a unique decomposition
$(U',L')$. By Lemma~\ref{lem:regularity},
$I_n+U'A$ is invertible. We claim $S_{U'}\succ0$. Granting this,
\eqref{eq:S_U_factor} gives $S_{U'}^{-1}(I_n+U'A)=(L'-U'^\top)^{-1}$, which is
lower triangular, so $U'$ satisfies the two conditions of part~\textbf{(i)} and
hence $U'=U$, making $I_n+UA$ invertible.

It remains to prove $S_{U'}\succ0$. We argue by induction on $n$, the case
$n=1$ being $S_{U'}=\pre_{11}>0$. Let $J=\{2,\dots,n\}$. By the
decoupling~\eqref{eq:trailing_block}, the trailing $J\times J$ block of
$S_{U'}$ is the matrix $S_{U'_{J,J}}$ formed from the positive definite pair
$(A_{J,J},\pre_{J,J})$, positive definite by the inductive hypothesis. By the
bordered-matrix criterion it suffices to show the Schur complement of this
block is positive. A direct computation using $\pre=L'+U'+U'AL'$ gives it as
\[
\pre_{11}-\pre_{1,J}\,(I_{n-1}+A_{J,J}\pre_{J,J})^{-1}A_{J,J}\,\pre_{J,1}.
\]
Since
$(I_{n-1}+A_{J,J}\pre_{J,J})^{-1}A_{J,J}
=\pre_{J,J}^{-1}-(\pre_{J,J}+\pre_{J,J}A_{J,J}\pre_{J,J})^{-1}\preceq\pre_{J,J}^{-1}$
(the subtracted matrix being positive definite), this is at least
$\pre_{11}-\pre_{1,J}\pre_{J,J}^{-1}\pre_{J,1}$, the Schur complement of
$\pre_{J,J}$ in the positive definite matrix $\pre$, which is positive. Hence
$S_{U'}\succ0$.
\end{proof}

\begin{example}
[The $2\times2$ case]
\label{exa:2x2}
Write $A=\bigl(\begin{smallmatrix}a&b\\b&c\end{smallmatrix}\bigr)$ and
$\pre=\bigl(\begin{smallmatrix}p&q\\q&r\end{smallmatrix}\bigr)$, and
$U=\bigl(\begin{smallmatrix}0 & u\\ 0 & 0\end{smallmatrix}\bigr)$. Running
Algorithm~\ref{alg:backward_elimination}: $L_{22}=r$, then $M_0=1+cr>0$ and
$u=q/(1+cr)$, so
\[
U=\begin{pmatrix}0 & q/(1+cr)\\ 0 & 0\end{pmatrix}.
\]
The pivot is $\pi_1=1+ub=(1+cr+bq)/(1+cr)$, which equals
$\det(I_n+UA)$. By Theorem~\ref{thm:triangular_decomposition}\textbf{(ii)} the
bilinear triangular decomposition exists if and only if
\[
1+cr+bq\neq 0,
\]
in which case
\[
L=\begin{pmatrix}\dfrac{(1+cr)p-cq^2}{1+cr+bq} & 0\\[6pt] q & r\end{pmatrix}.
\]

For $A=\bigl(\begin{smallmatrix}2 & -1\\ -1 & 1\end{smallmatrix}\bigr)$ and
$\pre=\bigl(\begin{smallmatrix}5 & 2\\ 2 & 1\end{smallmatrix}\bigr)$
(both positive definite with $\det A=\det\pre=1$), $1+cr+bq=1+1-2=0$, so the
bilinear triangular decomposition does not exist. The variational matrix of
Proposition~\ref{prop:variational_U} does: with
$U=\bigl(\begin{smallmatrix}0&1\\0&0\end{smallmatrix}\bigr)$,
$S_U=\bigl(\begin{smallmatrix}4&1\\1&1\end{smallmatrix}\bigr)\succ0$ and
$S_U^{-1}(I_n+UA)=\bigl(\begin{smallmatrix}0&0\\0&1\end{smallmatrix}\bigr)$ is
lower triangular.
\end{example}

\begin{remark}[The role of $A$]
\label{rem:good_A}
If $A$ is diagonal, then $(A,\pre)$ is triangularly regular for every positive
definite $\pre$: the product $UA$ is then strictly upper triangular, so
$I_n+UA$ is unipotent upper triangular and hence invertible, and
Proposition~\ref{prop:variational_U}\textbf{(ii)} applies. For non-diagonal $A$
regularity may fail, as the explicit $2\times2$ instance of
Example~\ref{exa:2x2} shows.
\end{remark}

\section{Conclusion}
\label{sec:conclusion}

In this work, we have shown that every positive definite $A$ admits a unique
decomposition precisely when the constraint subspace contains no nonzero
positive semidefinite matrix.

The decomposition is characterized as the unique minimizer of a single strictly
convex log-determinant functional. The same functional yields the dual
formulation of Section~\ref{subsec:duality_general} and the perturbation bounds
of Section~\ref{subsec:stability}, and the algorithms of
Section~\ref{sec:algorithms} minimize it directly, so a single object carries
both the theory and its computation.

As shown in Section~\ref{subsec:complexity_brief}, the algebraic structure of
the instance determines the computational cost, with several special cases, such
as group-invariant structure or a small primal or dual dimension, yielding
faster optimization. Instances in which the structure is retained by
$M(x)=A-C(x)$ at every iterate, such as banded and circulant matrices, are
particularly favorable.

One of the two applications developed here comes from mathematical finance and
the other is purely algebraic, which suggests that the same pattern occurs
elsewhere. While we have treated the log-determinant problem as a tool for a
linear-algebraic decomposition, the equivalent covariance estimation problem may
be of interest in its own right: the dual is the constrained Gaussian
maximum-likelihood problem for a precision matrix, so for a prescribed sparsity
or symmetry pattern the decomposition is the estimator, and the symmetry
inheritance of Section~\ref{sec:properties} and the structured costs of
Section~\ref{subsec:complexity_brief} apply to it unchanged.
Within mathematical finance, the exponential-utility problem of
Section~\ref{sec:finance} is one member of a larger family, and other Gaussian
hedging problems with information constraints should reduce to the same
decomposition. In linear algebra, the mechanism of
Section~\ref{sec:linear_algebra}, where an identity was obtained by applying the
theorem to an enlarged matrix with a suitable subspace, should yield further
identities.

Several directions seem natural. The theory should extend from real symmetric
matrices to complex Hermitian ones, and further to positive operators on
infinite-dimensional spaces. Replacing the subspace $\Sub$ by a general convex
constraint set would replace orthogonality by a normal-cone condition, though
existence and smooth dependence would then require different arguments. The
stability of the decomposition deserves a sharper description, both in the
regimes where $\kappa(B^{\ast},\Sub)$ or $\|B^{\ast}\|_2$ becomes large and
under perturbation of $\Sub$ itself, where a suitable distance from the
subspaces that violate the nondegeneracy condition would be the natural
parameter, pointing to regularization strategies for poorly conditioned
instances. Finally, in statistical applications $A$ is estimated from data,
and combining the perturbation bounds developed here with finite-sample
bounds for an estimator of $A$ would yield statistical guarantees for the
resulting decomposition.

\section*{Acknowledgments}

We would like to thank the referees for their careful reading of the manuscript and for their insightful comments and suggestions, which have significantly improved the quality and clarity of the paper.

\bibliographystyle{siamplain}
\bibliography{convex_bib}

\appendix


\section{Proofs of the perturbation bounds}
\label{app:perturbation}

\begin{proof}[Proof of Proposition~\ref{prop:perturbation_fro}]
Let $A_t:=A+t\Delta=(1-t)A+t\widetilde A$ for $t\in[0,1]$. Both $A$ and
$\widetilde A$ lie in $\SPD$, which is convex, so $A_t\in\SPD$, and
Theorem~\ref{thm:general_decomposition} supplies a decomposition $(B_t,C_t)$ of
$A_t$; thus $(B_0,C_0)=(B^{*},C^{*})$ and
$(B_1,C_1)=(\widetilde B,\widetilde C)$. Consider the optimality map
$F(A,C):=\proj_\Sub\big((A-C)^{-1}\big)\in \Sub$, where $\proj_\Sub$ is the
orthogonal projection of $\Sym$ onto $\Sub$, so
that $F(A_t,C_t)=0$ is the first-order condition
$\tr\bigl((A_t-C_t)^{-1}X\bigr)=0$ for all $X\in \Sub$.

\emph{Step 1 (differentiate along the path).} Since
$\frac{d}{ds}\big|_{s=0}(M+sE)^{-1}=-M^{-1}EM^{-1}$, differentiating $F$ in
each argument gives, for $(\Delta,X)\in\Sym\times \Sub$,
\begin{align}
\frac{d}{ds}\Big|_{s=0}F(A_t,C_t+sX)&=\proj_\Sub(B_tXB_t),\notag\\
\label{eq:dF}
\frac{d}{ds}\Big|_{s=0}F(A_t+s\Delta,C_t)&=-\proj_\Sub(B_t\Delta B_t).
\end{align}
By the definition of $\alpha$ and Lemma~\ref{lem:alpha_lipschitz}\textbf{(i)},
every $X\in \Sub\setminus\{0\}$ satisfies
\[
\langle\proj_\Sub(B_tXB_t),X\rangle_F=\tr(B_tXB_tX)\ge\alpha(B_t,\Sub)\|X\|_F^2>0 ,
\]
so the linear map $X\mapsto\proj_\Sub(B_tXB_t)$ is injective on the
finite-dimensional space $\Sub$, hence bijective.
Matrix inversion is real analytic on the invertible matrices, so $F$ is real
analytic, and with this invertibility the analytic implicit function theorem
applies: near every $A_t$ it
produces a real-analytic map solving $F(A,\cdot)=0$, which by the uniqueness in
Theorem~\ref{thm:general_decomposition} is $A\mapsto C^{*}(A)$. Hence
$t\mapsto C_t$ is real analytic, as is $t\mapsto B_t=(A_t-C_t)^{-1}$. Write
$\dot C_t:=\frac{d}{dt}C_t$ and $\dot B_t:=\frac{d}{dt}B_t$. Since every $C_t$
lies in the subspace $\Sub$, so does $\dot C_t$. By the chain rule
and~\eqref{eq:dF}, differentiating $F(A_t,C_t)=0$ in $t$ gives
\be
\label{eq:dC_eq}
\proj_\Sub\big(B_t\,\dot C_t\,B_t\big)=\proj_\Sub(B_t\Delta B_t).
\ee

\emph{Step 2 (the derivative is a projection).} For $B\in\SPD$ the formula
\be
\label{eq:B_inner}
\langle X,Y\rangle_B:=\tr(BXBY),\qquad \|X\|_B:=\sqrt{\tr(BXBX)} ,
\ee
defines an inner product on $\Sym$. It is bilinear and symmetric, and
$\langle X,X\rangle_B=\|B^{1/2}XB^{1/2}\|_F^2$ vanishes only for $X=0$, since
$B^{1/2}$ is invertible. Pairing~\eqref{eq:dC_eq} with $X\in \Sub$ and using
$\langle\proj_\Sub(Y),X\rangle_F=\langle Y,X\rangle_F$ for $Y\in\Sym$ shows that
$\langle \dot C_t,X\rangle_{B_t}=\langle\Delta,X\rangle_{B_t}$ for every
$X\in \Sub$. Since $\dot C_t$ itself lies in $\Sub$, this identifies $\dot C_t$
as the orthogonal projection of $\Delta$ onto $\Sub$ in the inner
product~\eqref{eq:B_inner}. Set $R_t:=\Delta-\dot C_t$. Orthogonality gives
$\|\Delta\|_{B_t}^2=\|\dot C_t\|_{B_t}^2+\|R_t\|_{B_t}^2$, so
$\|\dot C_t\|_{B_t}\le\|\Delta\|_{B_t}$ and $\|R_t\|_{B_t}\le\|\Delta\|_{B_t}$.
By~\eqref{eq:alpha_def_fro}, $\sqrt{\alpha(B_t,\Sub)}\|X\|_F\le\|X\|_{B_t}$ for
$X\in \Sub$, while $\|\Delta\|_{B_t}=\|B_t^{1/2}\Delta
B_t^{1/2}\|_F\le\|B_t\|_2\,\delta$, where $\delta:=\|\Delta\|_F$. Combining,
\be
\label{eq:dC_bound}
\|\dot C_t\|_F\le\frac{\|B_t\|_2}{\sqrt{\alpha(B_t,\Sub)}}\,\delta
=\kappa(B_t,\Sub)\,\delta .
\ee
Differentiating $B_t=(A_t-C_t)^{-1}$ gives $\dot B_t=-B_tR_tB_t$, and writing
$W:=B_t^{1/2}R_tB_t^{1/2}$, so that $B_tR_tB_t=B_t^{1/2}WB_t^{1/2}$, yields
\begin{align}
\|\dot B_t\|_2&\le\|B_t\|_2\|W\|_2\le\|B_t\|_2\|W\|_F
=\|B_t\|_2\|R_t\|_{B_t}\notag\\
\label{eq:dB_bound}
&\le\|B_t\|_2\|\Delta\|_{B_t}\le\|B_t\|_2^2\,\delta .
\end{align}

\emph{Step 3 (integrate).} Recall from
Lemma~\ref{lem:alpha_lipschitz}\textbf{(i)} that
$\alpha(B^{*},\Sub)\le\|B^{*}\|_2^2$, equivalently $\kappa(B^{*},\Sub)\ge1$. The
coefficients in~\eqref{eq:dC_bound} and~\eqref{eq:dB_bound} are evaluated at the
intermediate decomposition $B_t$, which is unknown, whereas the statement
involves only $B_0=B^{*}$. We therefore first bound $\|B_t\|_2$ in terms of
$\|B^{*}\|_2$.

Since $t\mapsto B_t$ is continuously differentiable on the compact interval
$[0,1]$, the map
$t\mapsto\|B_t\|_2$ is Lipschitz and hence absolutely continuous, and
by~\eqref{eq:dB_bound} it satisfies
$\frac{d}{dt}\|B_t\|_2\le\|\dot B_t\|_2\le\delta\,\|B_t\|_2^2$ for almost every
$t$. Therefore, by the chain rule, $\frac{d}{dt}\|B_t\|_2^{-1}\ge-\delta$
almost everywhere, and integrating from $0$ gives
$\|B_t\|_2^{-1}\ge\|B^{*}\|_2^{-1}-\delta t$. Hence, whenever
$\|B^{*}\|_2\,\delta<1$,
\be
\label{eq:B_comparison}
\|B_t\|_2\le\frac{\|B^{*}\|_2}{1-\|B^{*}\|_2\,\delta t},\qquad 0\le t\le 1.
\ee
By~\eqref{eq:smallness_fro} and $\kappa(B^{*},\Sub)\ge1$ we have
$\|B^{*}\|_2\,\delta=\varepsilon/\kappa(B^{*},\Sub)^2\le\varepsilon$.
Therefore~\eqref{eq:B_comparison} gives
$\|B_t\|_2\le\|B^{*}\|_2/(1-\varepsilon)$ for all $t\in[0,1]$, and
\begin{align*}
\|B_t-B^{*}\|_2&\le\int_0^t\delta\,\|B_s\|_2^2\,ds
\le\int_0^t\frac{\delta\,\|B^{*}\|_2^2}{\big(1-\|B^{*}\|_2\,\delta s\big)^2}\,ds\\
&=\frac{\|B^{*}\|_2^2\,\delta t}{1-\|B^{*}\|_2\,\delta t}
\le\frac{\|B^{*}\|_2^2\,\delta}{1-\|B^{*}\|_2\,\delta}\le\frac{\|B^{*}\|_2^2\,\delta}{1-\varepsilon} ,
\end{align*}
At $t=1$ the first bound is~\eqref{eq:B_stab_fro}.
Lemma~\ref{lem:alpha_lipschitz}\textbf{(ii)} bounds
the drift of $\alpha$, and~\eqref{eq:smallness_fro} in the form
$\|B^{*}\|_2^3\,\delta=\varepsilon\,\alpha(B^{*},\Sub)$ turns it into a fraction
of $\alpha(B^{*},\Sub)$,
\[
\big|\alpha(B_t,\Sub)-\alpha(B^{*},\Sub)\big|
\le\big(\|B_t\|_2+\|B^{*}\|_2\big)\|B_t-B^{*}\|_2
\le\frac{\varepsilon(2-\varepsilon)}{(1-\varepsilon)^2}\,\alpha(B^{*},\Sub) .
\]
Since $(1-\varepsilon)^2-\varepsilon(2-\varepsilon)=1-4\varepsilon+2\varepsilon^2$,
positive by the hypothesis on $\varepsilon$, this gives
$\alpha(B_t,\Sub)\ge(1-4\varepsilon+2\varepsilon^{2})(1-\varepsilon)^{-2}
\alpha(B^{*},\Sub)$ throughout. The factors $1-\varepsilon$ therefore cancel in
\[
\frac{\|B_t\|_2}{\sqrt{\alpha(B_t,\Sub)}}
\le\frac{\kappa(B^{*},\Sub)}{\sqrt{1-4\varepsilon+2\varepsilon^{2}}},
\]
and integrating~\eqref{eq:dC_bound} over $[0,1]$
yields~\eqref{eq:C_stab_fro}.
\end{proof}

\begin{proof}[Proof of Corollary~\ref{cor:perturbation_global}]
The maximum and minimum defining $\beta_0$ and $\alpha_0$ are attained because
$\Theta$ is compact, $A\mapsto B^{*}(A)$ is continuous, and $\alpha(\cdot,\Sub)$ is
continuous by Lemma~\ref{lem:alpha_lipschitz}\textbf{(ii)}. Moreover $\alpha_0>0$ by
Lemma~\ref{lem:alpha_lipschitz}\textbf{(i)}. For $A,\widetilde A\in\Theta$, the segment
$A_t=A+t(\widetilde A-A)$, $t\in[0,1]$, lies in $\Theta$ by convexity, and each
$A_t$ has a unique decomposition with $\|B_t\|_2\le\beta_0$ and
$\alpha(B_t,\Sub)\ge\alpha_0$, uniformly along the segment, which is what lets
the two bounds be integrated as they stand.
Integrating~\eqref{eq:dC_bound} gives
\[
\|C^{*}(\widetilde A)-C^{*}(A)\|_F\le\int_0^1\frac{\|B_t\|_2}{\sqrt{\alpha(B_t,\Sub)}}\|\widetilde A-A\|_F\,dt
\le\frac{\beta_0}{\sqrt{\alpha_0}}\|\widetilde A-A\|_F,
\]
and integrating~\eqref{eq:dB_bound} gives
$\|B^{*}(\widetilde A)-B^{*}(A)\|_2\le \beta_0^2\|\widetilde A-A\|_F$.
\end{proof}

\section{Algebraic examples and numerical illustrations}
\label{app:examples}
\label{subsec:algebraic_examples}

We illustrate the framework with six concrete examples, representing different
computational structural settings. The per-iteration cost of each is the corresponding row of Table~\ref{tab:per_iteration_complexity}.
\begin{enumerate}
\renewcommand{\labelenumi}{\Roman{enumi}.}
\item (dense, small $\Sub$, exact Newton): $A$ is a dense SPD matrix and $\Sub$ is spanned by $m$ symmetric matrices $\{D_1,\ldots,D_m\}$ with $m$ small.
The exact Newton method operates directly on the $m$-dimensional coordinate
 vector, each iteration dominated by a Cholesky factorization.

 \item (dense, small $\Sub$, Newton-CG): the setting of Example~I solved by
 Newton-CG rather than exact Newton.

 \item (dense, tridiagonal $\Sub^\perp$, dual Newton-CG): $A$ is dense SPD and $\Sub$ is
 high-dimensional but its orthogonal complement $\Sub^\perp$ has small dimension
 $m^\perp$, here taken to be tridiagonal (half-bandwidth $b=1$), so that $B^\ast$ is
 tridiagonal. The dual Newton method optimizes over the $m^\perp$ coordinates,
 while $M(x)$ stays dense.

 \item ($\calG$-invariant, $r$ blocks): $A$ is invariant under a group $\calG$
 of block permutations acting on $r$ blocks, and $\Sub$ is $\calG$-invariant with
 $C$ constrained to have zero diagonal.
 The problem reduces to the fixed-point subspace $\Sub^{\calG}$ of dimension
 $m_{\calG} \le r(r+1)/2$, and the per-iteration cost becomes independent of $n$.

 \item (banded): both $A$ and $\Sub$ have half-bandwidth $b \ll n$, so $M(x)$ is
 banded along the whole optimization path, a banded Cholesky factorization
 suffices at every iterate, and the linear algebra itself drops to $O(nb^2)$. A band-storage variant, reported as the
 indented row of Table~\ref{tab:scaling_summary}, stores only the band and so
 reaches much larger $n$.

 \item (circulant): $A$ and every $D_k$ are circulant and
 symmetric, so $M(x)$ stays circulant along the optimization path and is
 diagonalized by the discrete Fourier transform.
 \end{enumerate}

 \smallskip
Table~\ref{tab:scaling_summary} reports the measured per-iteration cost of
recovering the same pair $(B^\ast,C^\ast)$ in each setting, and
Figure~\ref{fig:scaling} its growth with $n$. What structure buys is best read
against the naive alternative for each setting, which differs from one to the
next. On the banded instance, a solver that keeps the banded factorization but
drops the structure-aware Hessian and trace evaluations is slower by a factor
of $6$ at $n=2{,}000$ and $20$ at $n=4{,}000$. Cost is reported per
{\it Newton iteration}, the quantity Table~\ref{tab:per_iteration_complexity} bounds.
For the Newton-CG settings such an iteration contains $N_{\mathrm{cg}}$
conjugate-gradient steps, a count that itself grows with $n$ in the dual
example. The dual setting is the most expensive in absolute terms
because it is the most demanding instance in the table: $M(x)$ is dense and the
dual dimension $m^\perp=2n-1$ grows with $n$, where the dense rows keep $m$
fixed. The dual formulation is what makes it tractable, since
$m\approx2\cdot10^6$ against $m^\perp=3{,}999$ at $n=2{,}000$, and
Table~\ref{tab:per_iteration_complexity} puts the primal cost on the same
instance roughly $330$ times higher.

\begin{table}[htbp]
\centering
\scriptsize
\renewcommand{\arraystretch}{1.15}%
\setlength{\tabcolsep}{4pt}%
\caption{Measured cost per Newton iteration across structural settings,
single-threaded, each entry the median over three random instances. The
per-iteration cost is separated from the one-time setup and output cost
by running each solve under two iteration caps and differencing. The circulant
entry is measured at $n=2{,}048$, the nearest power of two. The band-storage row has no
entry at $n=2{,}000$, being measured only at much larger $n$. The last two columns compare the fitted
$n$-exponent with the exponent of the bound in
Table~\ref{tab:per_iteration_complexity}. Every run in the $n=2{,}000$
column reaches a first-order residual below $10^{-8}$.}
\label{tab:scaling_summary}
\setlength{\tabcolsep}{1.5pt}
\begin{tabular}{@{}llrrrrrc@{}}
\toprule
& & \multicolumn{3}{c}{at $n=2{,}000$} & \multicolumn{3}{c}{over the swept range}\\
\cmidrule(lr){3-5}\cmidrule(lr){6-8}
Setting & Method
& dim & iters & sec/iter
& largest $n$ & fitted & theory\\
\midrule
Dense, small $\Sub$           & exact Newton    & $m=5$             & $2$  & $2.47$    & $3{,}000$   & $2.61$ & $3$ \\
Dense, small $\Sub$           & Newton-CG       & $m=5$             & $2$  & $5.86$    & $3{,}000$   & $2.83$ & $3$ \\
Dense, tridiag.\ $\Sub^\perp$ & dual Newton-CG  & $m^\perp=3{,}999$ & $15$ & $20.24$   & $3{,}000$   & $2.59$ & $3$ \\
$\calG$-invariant ($r=5$)     & Newton-CG       & $m_{\calG}=13$    & $5$  & $0.0034$  & $8{,}000$   & $0.00$ & $0$ \\
Banded ($b=2$)                & Newton-CG       & $m=3{,}997$       & $4$  & $0.181$   & $8{,}000$   & $1.02$ & $1$ \\
\quad\emph{band storage}      & Newton-CG       & $m=3{,}997$       & $4$  & ---       & $256{,}000$ & $1.04$ & $1$ \\
Circulant                     & Newton-CG       & $m=4$             & $6$  & $0.00027$ & $524{,}288$ & $1.08$ & $1$ \\
\bottomrule
\end{tabular}
\end{table}

\IfFileExists{fig_scaling_vs_n.pdf}{%
\begin{figure}[htbp]
\centering
\includegraphics[width=0.80\textwidth]{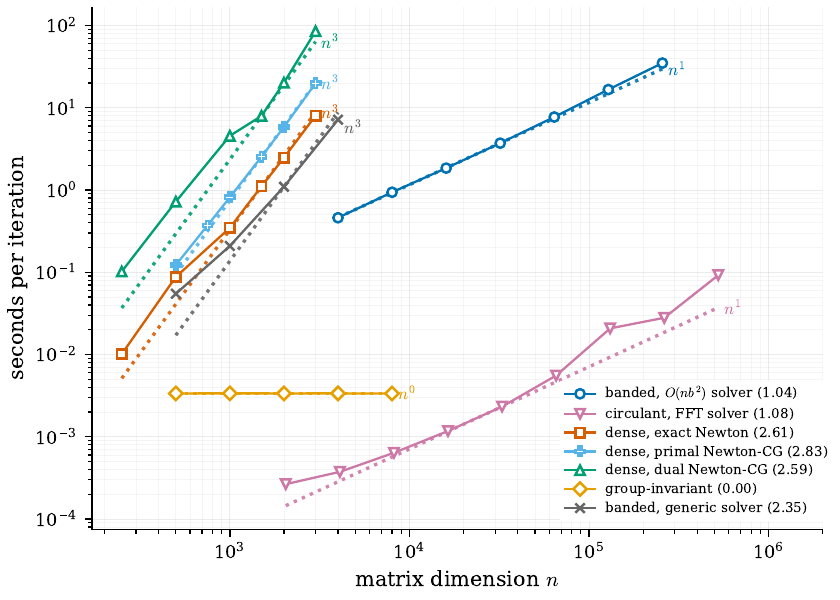}
\caption{Per-iteration cost against matrix dimension $n$ (log-scale), for the
structural settings of Table~\ref{tab:per_iteration_complexity}. Solid lines are
measurements; the legend gives their fitted $n$-exponents.
Dotted lines show the $n$-exponent of the corresponding bound,
anchored to the measured curve. Those bounds also carry a factor
$N_{\mathrm{cg}}$, so the dotted lines are a reference slope for the linear
algebra alone.}
\label{fig:scaling}
\end{figure}}{}

\paragraph{Reproducibility}
All timings used a single core of a Linux compute cluster with $32$\,GB of
memory per job. The BLAS thread count is pinned to one before NumPy is
imported, so that the measured slopes reflect the algorithms alone. The
implementation is
Python~3.12 with NumPy, SciPy and pandas. Dense instances are
$A=MM^{\top}+2nI_n$ with $M$ having independent standard normal entries, banded
instances use the same random construction restricted to a band of half-width
$b=2$, and circulant instances a random symmetric first column supported on six
lags. The seeds are $0,1,2$.

The constraint dimension grows with $n$ as $m^{\perp}=2n-1$ in the dual setting
and $m=2n-3$ in the banded ones, and is fixed in the others. The iteration caps used for the differencing are $4$ and $12$
($2$ and $8$ for dense Newton-CG), with the convergence test disabled so that
exactly the requested number of iterations is taken. The fitted exponents are least-squares slopes of the logarithm
of the time per iteration against $\log n$, over the whole $n$ grid of each
setting. The residual quoted in Table~\ref{tab:scaling_summary} is the largest
absolute component of the gradient at the returned iterate, namely
$\max_k|\tr(BD_k)|$ in the primal settings and the corresponding dual gradient
for Example~II.
The driver scripts, the $n$ grids and the raw measurements are available in the
repository cited in Section~\ref{sec:optimization}.

Figure~\ref{fig:block_group_demo} illustrates the block-permutation invariant
setting of Example~IV on a small instance with $n=25$ (five blocks of size
five), chosen for visualization. It shows how symmetry reduces the effective
optimization dimension while enforcing exact structural constraints in both $C$
and $B$.

\begin{figure}[htbp]
\centering
\includegraphics[width=0.84\textwidth]{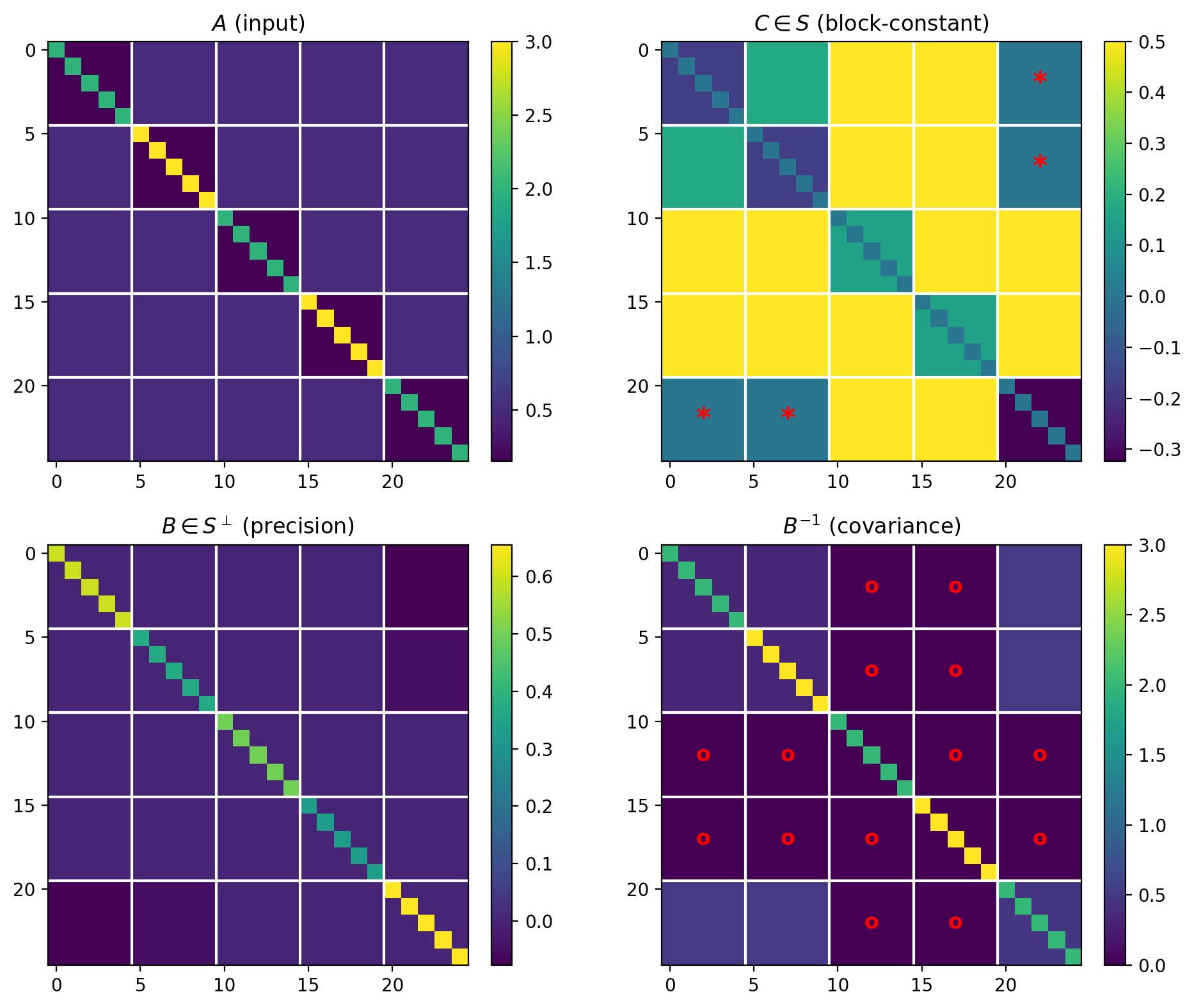}
\caption{Block-permutation invariant decomposition (Example~IV), with $n=25$ (five blocks of size five).
Top left: input matrix $A$.
Top right: constraint matrix $C\in \Sub$, block-constant on symmetric block pairs.
Red stars mark the block pairs excluded from $\Sub$, on which $C$ vanishes.
Bottom left: precision matrix $B\in \Sub^\perp$.
Bottom right: covariance matrix $B^{-1}$, with red circles on the block pairs where it vanishes exactly.}
\label{fig:block_group_demo}
\end{figure}

\clearpage
\newpage
\section{Newton-CG implementation details}
\label{app:newtoncg_pseudocode}

The Newton-CG routines of
Section~\ref{sec:optimization} are shown below: Algorithm~\ref{alg:newtoncg} runs the outer
Newton iteration with inner conjugate-gradient steps,
Algorithm~\ref{alg:evaluate_phi_grad} evaluates $\Phi$, its gradient,
and the inverse $B=M^{-1}$ of $M=A-C(x)$, and
Algorithm~\ref{alg:hv_product} computes the
Hessian-vector products~\eqref{eq:Hv_formula}. They are written for the dense
case; otherwise structure-exploiting solvers for $M$ may replace $B$.

\vspace{-0.2cm}
\begin{algorithm}[H]
\caption{\textproc{ConstrainedDecompositionNewtonCG}}
\label{alg:newtoncg}
\begin{algorithmic}[1]
\Require $A\in\SPD$, basis $\{D_k\}_{k=1}^m$ of $\Sub$
\Parameters tolerances $\texttt{tol},\texttt{tol\_cg}>0$
\Ensure The minimizer $x$, $C^\ast=C(x)$, and $B^\ast=(A-C^\ast)^{-1}$
\State $x\gets0$, \; $(\Phi,g,B)\gets\Call{EvaluatePhiGrad}{A,x}$
\While{$\|g\|_2>\texttt{tol}$}
 \State $d\gets0$, \; $r\gets-g$, \; $s\gets r$
 \While{$\|r\|_2>\min\{\texttt{tol\_cg},\|g\|_2\}\,\|g\|_2$}
  \State $q\gets\Call{HvProduct}{\{D_k\},B,s}$
  \State $\theta\gets\|r\|_2^2/(s^\top q)$, \; $d\gets d+\theta s$, \; $r^{+}\gets r-\theta q$
  \State $s\gets r^{+}+\bigl(\|r^{+}\|_2^2/\|r\|_2^2\bigr)s$, \; $r\gets r^{+}$
 \EndWhile
 \State Choose $t$ by feasibility-preserving Armijo backtracking, \; $x\gets x+td$
 \State $(\Phi,g,B)\gets\Call{EvaluatePhiGrad}{A,x}$
\EndWhile
\State $C^\ast\gets C(x)$, \; $B^\ast\gets B$
\State \Return $x$, $C^\ast$, $B^\ast$
\end{algorithmic}
\end{algorithm}



\vspace{-0.3cm}
\begin{algorithm}[H]
\caption{\textproc{EvaluatePhiGrad}}
\label{alg:evaluate_phi_grad}
\begin{algorithmic}[1]
\Require $A\in\SPD$, $x\in\Omega_x$
\Ensure $\Phi$, the gradient $g=\nabla\Phi(x)\in\R^m$, and $B=M^{-1}$
\State $C\gets\sum_{k=1}^m x_k D_k$, \; $M\gets A-C$
\State Factorize $M$, \; $\Phi\gets-\log\det(M)$, \; $B\gets M^{-1}$
\For{$k=1$ \textbf{to} $m$}
 \State $g_k\gets\tr(BD_k)$
\EndFor
\State \Return $\Phi$, $g$, $B$
\end{algorithmic}
\end{algorithm}



\vspace{-0.4cm}
\begin{algorithm}[H]
\caption{\textproc{HvProduct}}
\label{alg:hv_product}
\begin{algorithmic}[1]
\Require Basis $\{D_k\}_{k=1}^m$, $B\in\SPD$, a vector $p\in\R^m$
\Ensure $q=H(x)p\in\R^m$
\State $D(p)\gets\sum_{\ell=1}^m p_\ell D_\ell$
\State $Y\gets B\,D(p)\,B$
\For{$k=1$ \textbf{to} $m$}
 \State $q_k\gets\tr(D_kY)$
\EndFor
\State \Return $q$
\end{algorithmic}
\end{algorithm}

\end{document}